\documentclass{article}

\usepackage{graphicx}
\usepackage{amsmath,amsthm}
\usepackage{amssymb,bm}
\usepackage{amsfonts}
\usepackage{latexsym}
\usepackage{mathrsfs}
\usepackage{mathtools}
\usepackage[page]{appendix}
\usepackage[longend,ruled,linesnumbered]{algorithm2e}
\usepackage{changes}
\usepackage{changes}
\usepackage{color}

\DeclareMathOperator*{\argmin}{argmin}

\newcommand{\abs}[1]{|#1|}
\newcommand{\Inner}[2]{\left\langle #1,#2\right\rangle}
\newcommand{\inner}[2]{\langle #1,#2\rangle}
\newcommand{\N}{\mathbb{N}}  
\newcommand{\Nc}{\mathbf{N}} 
\newcommand{\Norm}[1]{\left\|{#1}\right\|}
\newcommand{\norm}[1]{\|{#1}\|}

\newcommand{\R}{\mathbb{R}}  
\newcommand\set[1]{\{#1\}}
\newcommand\seq[1]{\left(#1\right)}
\newcommand\Set[1]{\left\{#1\right\}}
\newcommand{\tos}{\rightrightarrows}

\newcommand{\Lb}{\mathscr{L}}
\newcommand{\So}{\mathbf{S}}
\newcommand{\sol}{\mathcal{K}}

\newcommand{\bpt}[1]{\mathring{#1}}

\newtheorem{theorem}{Theorem}[section]

\newtheorem{definition}[theorem]{Definition}

\newtheorem{lemma}[theorem]{Lemma}
\newtheorem{proposition}[theorem]{Proposition}

\title{Iteration-complexity of a Rockafellar's 
proximal method of multipliers for convex
programming based on second-order approximations}

\author{M. Marques Alves\thanks{%
              Departamento de Matem\'atica,
              Universidade Federal de Santa Catarina,
              Florian\'opolis, Brazil, 88040-900.
               email: maicon.alves@ufsc.br.
              This author was partially 
              supported by CNPq grants no.
              406250/2013-8, 237068/2013-3 and 306317/2014-1.}
\and
R.D.C. Monteiro\thanks{
              School of Industrial and Systems
    Engineering, Georgia Institute of
    Technology, Atlanta, GA, 30332-0205.
    email: monteiro@isye.gatech.edu.
		This author
    was partially supported by CNPq Grant 406250/2013-8 and NSF Grant CMMI-1300221.
}
\and
Benar F. Svaiter\thanks{%
IMPA, Estrada Dona Castorina 110, 22460-320 
Rio de Janeiro, Brazil. 
          email: benar@impa.br
    This author was partially supported by CNPq grants
    474996/2013-1, 302962/2011-5, FAPERJ grant E-26/102.940/2011,
    and PRONEX-Optimization.}
}

\begin{document}

\maketitle

\begin{abstract}
  This paper studies the iteration-complexity of a new primal-dual algorithm
  based on Rockafellar's proximal method of multipliers (PMM) for solving
  smooth convex programming problems with inequality constraints. In each
  step, either a step of Rockafellar's PMM for a second-order model of the
  problem is computed or a relaxed extragradient step is performed. The
  resulting algorithm is a (large-step) relaxed hybrid proximal
  extragradient (r-HPE) method of multipliers, which combines Rockafellar's
  PMM with the r-HPE method.
 \\
 \\
\textsc{keywords:}
convex programming, proximal method of multipliers, second-order methods,
 hybrid extragradient, complexity
\end{abstract}

\section{Introduction}
\label{intro}

The smooth convex programming problem 
with (for the sake of simplicity)
only inequality constraints is
\begin{align}
  \label{eq:np.i}
  \begin{array}{ll} 
    \min& f(x)\qquad 
    \mathrm{s.t.}\;\; g(x)\leq 0
  \end{array}
\end{align} 
where $f: \R^n \to \R$ and the components of
$g=(g_1,\ldots,g_m): \R^n \to \R^m$ are 
smooth convex functions.
Dual methods for this problem 
solve the associated dual problem
\begin{align*}
  \max\;\bigg(\inf_{x\in\R^n} f(x)+\inner{y}{g(x)}\bigg)\qquad
   \mathrm{s.t.}\;\; y\geq 0
\end{align*}
and, \emph{en passant}, find a solution of
the original (primal) problem. 
Notice that a pair $(x,y)$ 
satisfies the Karush-Kuhn-Tucker conditions for
problem \eqref{eq:np.i} 
if and only if $x$ is a solution of this
problem, $y$ is a solution of the associated 
dual problem, and there is no
duality gap.

The \emph{method of multipliers}, 
which was proposed by
Hestenes~\cite{hes-1,hes-2} and Powel~\cite{pow-1} 
for equality
constrained optimization problems and extended by
Rockafellar~\cite{roc-mul} 
(see also~\cite{MR0418919}) to inequality
constrained convex programming problems, 
is a typical example of a dual
method.
It generates iteractively sequences $(x_k)$ and $(y_k)$ as
follows:
\begin{align*}
  x_k\approx \arg\min_{x\in\R^n} \;\; \mathscr{L}(x,y_{k-1},\lambda_k),
  \qquad y_k=y_{k-1}+\lambda_k\nabla_y \mathscr{L}(x_{k},y_{k-1},\lambda_k)
\end{align*}
where $\approx$ stands for approximate solution, $\lambda_k>0$, and
$\mathscr{L}(x,y,\lambda)$ is the \emph{augmented Lagrangian}
\begin{align*}
  \mathscr{L}(x,y,\lambda)&=f(x)+\dfrac{1}{2\lambda}
  [\norm{(y+\lambda  g(x))_+}^2-\norm{y}^2]\\
  &=\max_{y'\geq 0}\;\;f(x)+\inner{y'}{g(x)}-\dfrac{1}{2\lambda}\norm{y'-y}^2.
\end{align*}
The method of multipliers is also 
called the \emph{augmented Lagrangian
method}.
In the seminal article~\cite{MR0418919}, 
Rockafellar proved that the
method of multipliers is an instance of his proximal point
method (hereafter PPM)~\cite{Rock:ppa} applied to the dual objective function.
Still
in~\cite{MR0418919}, Rockafellar 
proposed a new primal-dual method for
\eqref{eq:np.i}, which we discuss next 
and that  we will use in this paper to 
design a new primal-dual method for this problem.

Rockafellar's \emph{proximal method of multipliers} (hereafter PMM)~\cite{MR0418919}
generates, for any starting point $(x_0,y_0)$, a sequence
$\big((x_k,y_k)\big)_{k\in\N}$ 
as
the approximate solution of a 
regularized saddle-point problem
%
\begin{align}
  \label{eq:pmm}
  (x_k,y_k) \approx \mbox{arg}
  & \min_{x\in\R^n}\max_{y\in\R^m_+}\;f(x)+
  \inner{y}{g(x)}
  + \dfrac{1}{2\lambda}
    \left[\norm{x-\bpt{x}}^2-\norm{y-\bpt{y}}^2\right]
\end{align}
where 
$(\bpt{x},\bpt{y})=(x_{k-1},y_{k-1})$
is the current iterate and
$\lambda=\lambda_k>0$ is a stepsize parameter.
Notice that the objective function of 
the above saddle-point problem
is obtained by adding 
to the augmented Lagrangian a proximal
term for the primal variable $x$.
If $ \inf \lambda_k>0$ and
\begin{align} \label{eq:summ-error}
  \sum_{k=1}^{\infty}\norm{(x_k,y_k)-(x^*_k,y^*_k)} < \infty
\end{align}
where $(x^*_k,y^*_k)$ is the (exact) solution of
\eqref{eq:pmm},
then $\seq{(x_k,y_k)}_{k\in \N}$ converges to a solution of 
the Karush-Kuhn-Tucker conditions for~\eqref{eq:np.i}
provided that 
there exist a pair satisfying these conditions.
%
This result follows from the facts that the
satisfaction of KKT conditions 
for~\eqref{eq:np.i} 
can be
formulated as a monotone inclusion problem 
and~\eqref{eq:pmm}
is the Rockafellar's 
PPM iteration 
for this inclusion problem 
(see comments after Proposition~\ref{pr:rf.rsdp}).
%
Although~\eqref{eq:pmm} is a (strongly)
convex-concave problem -- 
and hence has a unique solution -- the computation
of its exact or an approximate solution can be very hard.

We assume in this paper 
that $f$ and $g_i$ ($i=1,\ldots, m$) 
are $\mathscr{C}^2$
convex functions with Lipschitz continuous Hessians.
The method proposed in this paper either
solves a second-order model
of \eqref{eq:pmm} in which
second-order approximations
of $f$ and $g_i$ ($i=1,\dots, m$)
replace these functions in \eqref{eq:pmm}
or performs a (relaxed) extragradient step.
In its general form, PMM is an inexact PPM in that each iteration
approximately solves \eqref{eq:pmm} according to the summable error
criterion \eqref{eq:summ-error}.
The method proposed in this paper can also be viewed as an inexact PPM but
one based on a relative error criterion instead of the one in
\eqref{eq:summ-error}. More specifically, it can be viewed as an instance of
the (large-step) relaxed hybrid proximal extragradient (r-HPE) method
\cite{MonSva11-1,Sol-Sv:hy.ext,pre-print-benar} which we briefly discuss
next.

Given a point-to-set maximal monotone operator
$T:\R^p\tos \R^p$, the large-step r-HPE method 
computes approximate solutions for the monotone
inclusion problem $0\in T(z)$ as extragradient steps
\begin{align}
\label{eq:es.intr}
 z_{k}=z_{k-1}-\tau \lambda_k v_k,
\end{align}  
where $z_{k-1}$ is the current iterate,
$\tau\in (0,1]$
is a relaxation parameter, 
$\lambda_k>0$ is the stepsize 
and $v_k$ together with the pair
$(\tilde z_k,\varepsilon_k)$ satisfy 
the following
conditions
\begin{align}
\label{eq:ec.intr}
\begin{aligned}
 &v_k\in T^{[\varepsilon_k]}(\tilde z_k),\qquad
\norm{\lambda_k v_k+\tilde z_k-z_{k-1}}^2+
2\lambda_k\varepsilon_k
\leq \sigma^2\norm{\tilde z_k-z_{k-1}}^2,\\
&\lambda_k\norm{\tilde z_k-z_{k-1}}\geq \eta
\end{aligned}
\end{align}
where $\sigma\in [0,1)$ and $\eta>0$ are given constants and $T^\varepsilon$
denotes the $\varepsilon$-enlargement of $T$. (It has the property that
$T^\varepsilon(z) \supset T(z)$ for every $z$.)
%
The method proposed in this paper for solving the minimization
problem~\eqref{eq:np.i} can be viewed as a realization of the above
framework where the operator $T$ is the standard saddle-point operator
defined as
$T(z):=(\nabla f(x)+\nabla g(x)y,-g(x)+N_{\R^m_{+}}(y))$ for every
$z=(x,y)$.
More specifically, the method consists of two type of iterations. The ones
which perform extragradient steps can be viewed as a realization of
\eqref{eq:ec.intr}.
On the other hand, each one of the other iterations updates the stepsize by
increasing it by a multiplicative factor larger than one and then solves a
suitable second-order model of \eqref{eq:pmm}. After a few of these
iterations, an approximate solution satisfying \eqref{eq:ec.intr} is then
obtained. Hence, in contrast to the PMM which does not specify how to obtain
an approximate solution $(x_k,y_k)$ of \eqref{eq:pmm}, or equivalently the
prox inclusion $0 \in \lambda_k T(z) + z-z_{k-1}$ with $T$ as above, these
iterations provide a concrete scheme for computing an approximate solution
of this prox inclusion according to the relative criterion in
\eqref{eq:ec.intr}.
Pointwise and ergodic iteration-complexity bounds are then derived for
our method using the fact that the large-step r-HPE method
has pointwise and ergodic global convergence rates
of $\mathcal{O}(1/k)$ and 
$\mathcal{O}(1/k^{3/2})$, respectively.

The paper is organized as follows.
Section~\ref{sec:sqp} reviews 
some basic properties
of $\varepsilon$-enlargements of 
maximal monotone operators and
briefly reviews the basic properties of 
PPM and the large-step
r-HPE method.
Section~\ref{sec:sqp2} 
presents the basic properties 
of the minimization problem of interest
and some equivalences between 
certain saddle-point, complementarity
and monotone inclusion problems, as well
as of its regularized versions.  
Section~\ref{sec:em.pp} introduces an error measure, shows some of its
properties and how it is related to the relative error criterion for the
large-step r-HPE method.
Section~\ref{sec:qa} studies the smooth convex programming
problem~\eqref{eq:np.i} and its second-order approximations.
The proposed method 
(Algorithm 1) is presented in Section~\ref{sec:sqr-hpemm}
and its iteration-complexities (pointwise and ergodic)
are studied in Section~\ref{sec:rhpe}.

\section{Rockafellar's proximal point method 
and the hybrid proximal extragradient method}
\label{sec:sqp}

This work is based on Rockafellar's proximal point method (PPM). The new
method presented in this paper is a particular instance of the (large-step)
relaxed hybrid proximal extragradient (r-HPE) method \cite{pre-print-impa}.
For these reasons, in this section we review Rockafellar's PPM, the
large-step r-HPE method, and review some convergence properties of these
methods.
\subsection*{Maximal monotone operators, the monotone inclusion problem,
and Rockafellar's proximal point method}

A \emph{point-to-set operator} in ${\R^p}$, $T:{\R^p}\tos{\R^p}$, is a
relation $T\subset {\R^p}\times{\R^p}$ and
\[
T(z):=\{v\;|\; (z,v)\in T\},\qquad z\in {\R^p}.
\]
The \emph{inverse} of $T$ is $T^{-1}:{\R^p}\tos{\R^p}$,
$T^{-1}:=\{(v,z)\;|\; (z,v)\in T\}$. The \emph{domain} and the range of
$T$ are, respectively,
\[
D(T):=\set{z\;|\; T(z)\neq\emptyset},\quad R(T):=\set{v\;|\;\exists z\in
  {\R^p},\; v\in T(z)}.
\]
When $T(z)$ is a singleton for all $z\in D(T)$, it is usual to identify $T$
with the map $D(T)\ni z \mapsto v\in{\R^p}$ where $T(z)=\set{v}$.
If $T_1,T_2:{\R^p}\tos{\R^p}$ and $\lambda\in\R$, then
$T_1+T_2:{\R^p}\tos{\R^p}$ and $\lambda T_1:{\R^p}\tos{\R^p}$ are defined as
\[
(T_1+T_2)(z):=\set{v_1+v_2\;|\; v_1\in T_1(z),\;v_2\in T_2(z)},\qquad
(\lambda T_1)(z):=\set{\lambda v\;|\; v\in T_1(z)}.
\]

A point-to-set operator $T:{\R^p}\tos{\R^p}$ is \emph{monotone} if
\[
\inner{z-z'}{v-v'}\geq 0,\qquad \forall (z,v),(z',v')\in T
\]
and it is \emph{maximal monotone} if it is a maximal element in the family
of monotone point-to-set operators in ${\R^p}$ with respect to the partial
order of set inclusion.
The subdifferential of a proper closed convex function is a classical
example of a maximal monotone operator.
Minty's theorem~\cite{MR0169064} states that if $T$ is maximal monotone
and $\lambda>0$, then the 
\emph{proximal map} $(\lambda T+I)^{-1}$ is a
point-to-point nonexpansive operator with domain ${\R^p}$.

The \emph{monotone inclusion problem} is: given $T:{\R^p}\tos{\R^p}$ maximal
monotone, find $z$ such that
\begin{align}
  \label{eq:g.mip}
  0\in T(z).  
\end{align}
Rockafellar's PPM \cite{Rock:ppa}
generates, for any starting $z_0\in{\R^p}$, a sequence $(z_k)$ by the
approximate rule
\begin{align*}
  z_k\approx (\lambda_k T+I)^{-1}z_{k-1},
\end{align*}
where $(\lambda_k)$ is a sequence of strictly positive \emph{stepsizes}.
Rockafellar proved~\cite{Rock:ppa} that if \eqref{eq:g.mip} has a
solution and
\begin{align}
  \label{eq:et.ppm}
  \Norm{z_k-(\lambda_k T+I)^{-1}z_{k-1}}\leq e_k,\;
   \sum_{k=1}^{\infty}e_k<\infty,\;\;\;
\inf \lambda_k>0,
\end{align}
then $(z_k)$ converges to a solution of \eqref{eq:g.mip}.

In each step of the PPM, computation of the proximal map
$(\lambda T+I)^{-1}z$ amounts to solving the \emph{proximal (sub)
  problem}
\begin{align*}
  0\in \lambda T(z_+)+z_+-z,
\end{align*}
a regularized inclusion problem which, although well-posed,
is almost as hard as \eqref{eq:g.mip}.
From this fact stems the necessity of using approximations of the proximal
map, for example, as prescribed in \eqref{eq:et.ppm}.
Moreover, since each new iterate is, hopefully, just a better approximation
to the solution than the old one, if it was
computed with high accuracy, then the computational cost of each iteration
would be too high (or even prohibitive) and this would impair the overall
performance of the method (or even make it infeasible).

So, it seems natural to try to improve Rockafellar's PPM by devising a
variant of this method that would accept a \emph{relative} error tolerance
and wherein the progress of the iterates towards the solution set could be
estimated.
In the next subsection we discuss the hybrid proximal extragradient (HPE)
method, a variant of the PPM which aims to satisfy these goals.

\subsection*{Enlargements of maximal monotone 
operators and the hybrid
proximal extragradient method}

The HPE method
\cite{Sol-Sv:hy.ext,Sol-Sv:hy.proj} is a modification of
Rockafellar's PPM wherein:
(a) the proximal subproblem, in each
iteration, is to be solved within a 
\emph{relative} error tolerance and
(b) the update rule is modified so as to guarantee that the next iterate
is closer to the solution set by a quantifiable amount.

An additional feature of (a) is that, in some sense, errors in the
inclusion on the proximal subproblems are allowed.
Recall that the \emph{$\varepsilon$-enlargement}~\cite{Bu-Iu-Sv:teps} of
a maximal monotone operator $T:{\R^p}\tos{\R^p}$ is
\begin{align}
  \label{eq:teps}
  T^{[\varepsilon]}(z):=\set{v\;|\; \inner{z-z'}{v-v'}\geq-\varepsilon
  \;\forall(z',v')\in T},\qquad 
  z\in{\R^p},\;\varepsilon\geq0.
\end{align}
From now on in this section $T:{\R^p}\tos{\R^p}$ is a maximal monotone
operator. The r-HPE
method~\cite{Sol-Sv:hy.unif} for the monotone inclusion
problem~\eqref{eq:g.mip} proceed as follows: 
choose ${z}_0\in{\R^p}$ and $\sigma\in [0,1)$; 
for $i=1,2,\dots$ compute $\lambda_i>0$
and $(\tilde z_i,v_i,\varepsilon_i)\in \R^p\times \R^p\times \R_+$
such that
\begin{align}
  \label{eq:rhpe}
  \begin{aligned}
    & v_i\in T^{[\varepsilon_i]}(\tilde z_i),\;\ 
  \norm{{\lambda}_i v_i+\tilde
      {z}_i-{z}_{i-1}}^2+2{\lambda}_i\varepsilon_i
    \leq \sigma^2\norm{\tilde {z}_i-{z}_{i-1}}^2,\\
		& \text{choose }\tau_i\in(0,1]\text{ and
      set}\\
		&{z}_i={z}_{i-1}-\tau_i{\lambda}_i v_i.
  \end{aligned}
\end{align}
In practical applications, 
 each problem has a particular structure which
may render feasible the computation of 
$\lambda_i$, $\tilde z_i$, $v_i$,
and $\varepsilon_i$ as prescribed above.
For example, $T$ may be Lipschitz continuous, it may be differentiable,
or it may be a sum of an operator which has a proximal map easily
computable with others with some of these properties.
Prescriptions for computing 
$\lambda_i$, $\tilde z_i$, $v_i$,
and $\varepsilon_i$ under 
each one of these assumptions were presented
in~\cite{MonSva10-1,MonSva10-3,MonSva11-1,MonSva10-2,sol.sva-glo.ao99,Sol-Sv:hy.ext,MR1740961,MR1953827}.

Computation of  $(\lambda T+I)^{-1}z$ is equivalent to
the resolution of  an
inclusion-equation system:
\[
z_+=(\lambda T+I)^{-1}z\iff  \exists v\in T(z_+),\;\lambda v+z_+-z=0.
\]
Whence, the error criterion in the 
first line of~\eqref{eq:rhpe} relaxes
both the inclusion and the equality at the 
right-hand side of the above
equivalence.
Altogether, each r-HPE iteration consists in:
(1) solving (with a relative error tolerance) a 
``proximal''
inclusion-equation system;
(2) updating $z_{i-1}$ to $z_i$ by means of an 
extragradient step, that is,
using $v_i\in T^{[\varepsilon_i]}(\tilde z_i)$. 

In the remainder part of this section we 
present some convergence properties of the
r-HPE method which were essentially proved in 
\cite{pre-print-impa} and revised
in~\cite{pre-print-benar}.
The next proposition shows that 
$z_i$ is closer than $z_{i-1}$ to
the solution set with respect to the square of the 
norm, 
by a quantifiable
amount, and present some useful estimations.

\begin{proposition}
 [{\cite[Proposition 2.2]{pre-print-benar}}]
  \label{pr:a1}
  For any $i\geq 1$ and $z^*\in T^{-1}(0)$,
  \begin{enumerate}
  \item[\emph{(a)}]
    \label{it:1-b}
    $(1-\sigma)\norm{\tilde z_i-z_{i-1}} 
   \leq \norm{\lambda_i v_i} \leq
    (1+\sigma)\norm{\tilde z_i-z_{i-1}}$
    and $2\lambda_i\varepsilon_i 
   \leq \sigma^2 \norm{\tilde z_i-z_{i-1}}^2$;
  \item[\emph{(b)}]
    \label{it:2-b}
    $\norm{z^*-z_{i-1}}^2 \geq \norm{z^*-z_i}^2+
    \tau_i(1-\sigma^2)\norm{\tilde
      z_i-z_{i-1}}^2\geq\norm{z^*-z_i}^2$;
  \item[\emph{(c)}]
    \label{it:3-b}
    $\norm{z^*-z_0}^2 \geq 
   \norm{z^*-z_i}^2+(1-\sigma^2)\sum_{j=1}^{i}
    \tau_j\norm{\tilde z_j-z_{j-1}}^2$;
  \item[\emph{(d)}] \label{it:4-b}
    $\norm{z^*-\tilde z_i}\leq
   \norm{z^*-z_{i-1}}/\sqrt{1-\sigma^2}$
    and $\norm{\tilde z_i-z_{i-1}} 
   \leq \norm{z^*-z_{i-1}}/\sqrt{1-\sigma^2}$.
  \end{enumerate}
 \end{proposition}

The aggregate stepsize 
$\Lambda_i$ and the ergodic sequences
$(\tilde {z}_i^a)$, $(\tilde v_i^a)$, 
$(\varepsilon_i^a)$ associated
with the sequences $(\lambda_i)$, 
$(\tilde {z}_i)$,  $(v_i)$, and
$(\varepsilon_i)$ are, respectively,
\begin{align}
\label{eq:d.eg}
  \begin{aligned}
    &\Lambda_i:=\sum_{j=1}^i \tau_j{\lambda}_j,\\
    &\tilde {z}_i^{\,a}:= \frac{1}{\;\Lambda_i}\;
   \sum_{j=1}^i\tau_j {\lambda}_j
    \tilde {z}_j, \quad 
   v_i^{\,a}:= \frac{1}{\;\Lambda_i}\;\sum_{j=1}^i
    \tau_j{\lambda}_j v_j,\\ 
   &\varepsilon_i^{\,a}:=
    \frac{1}{\;\Lambda_i}\;\sum_{j=1}^i 
    \tau_j{\lambda}_j (\varepsilon_j
    +\inner{\tilde {z}_j-
     \tilde {z}_i^{\,a}}{v_j-v_i^{\,a}}).
  \end{aligned}
\end{align}

Next we present the pointwise and 
ergodic iteration-complexities
of the large-step r-HPE method, i.e.,
the r-HPE method with a large-step condition 
\cite{MonSva10-3,MonSva11-1}.
We also assume that
the sequence of relaxation parameters
$(\tau_i)$ is bounded away from zero.

\begin{theorem}[{\cite[Theorem 2.4]{pre-print-benar}}]
  \label{lm:rhpe2}
  If $d_0$ is the distance from 
  $z_0$ to $T^{-1}(0)\neq\emptyset$
  and 
  \begin{align*}
    \lambda_i\norm{\tilde z_i-z_{i-1}} \geq
    \eta>0,\quad \tau_i\geq \tau>0
    \qquad\qquad i=1,2,\ldots
  \end{align*}
  then, for any $i\geq1$,
  \begin{enumerate}
  \item[\emph{(a)}]
    there exists $j\in\set{1,\dots,i}$ 
    such that $v_j\in T^{[\varepsilon_j]}(\tilde z_j)$ 
    and 
    \begin{align*}
      \norm{v_j}\leq\dfrac{d_0^2}{i\tau(1-\sigma)\eta},\qquad
      \varepsilon_j\leq
      \dfrac{\sigma^2d_0^3}
       {(i\tau)^{3/2}(1-\sigma^2)^{3/2}2\eta};
    \end{align*}
  \item[\emph{(b)}]
    $v_i^a\in T^{[\varepsilon_i^a]}(\tilde z_i^a)$,
    $\norm{v_i^a}\leq 
    \dfrac{2d_0^2}{(i\tau)^{3/2}
      (\sqrt{1-\sigma^2})\eta}$,
      and $\varepsilon_i^a\leq
    \dfrac{2d_0^3}{(i\tau)^{3/2}(1-\sigma^2)\eta}$.
  \end{enumerate}
\end{theorem}

\noindent
{\bf Remark.} 
We mention that the inclusion in Item (a) of
Theorem \ref{lm:rhpe2} is in the 
enlargement of $T$ 
which appears in the inclusion in \eqref{eq:rhpe}.
To be more precise, in some applications the operator
$T$ may have a special structure, like for instance
$T=S+N_{\mathcal{X}}$, where $S$ is point-to-point
and $N_{\mathcal{X}}$ is the normal cone operator 
of a closed convex
set $\mathcal{X}$, and the inclusion in
\eqref{eq:rhpe}, in this case, is 
$v_i\in \left(S+N_{\mathcal{X}}^{[\varepsilon_i]}
\right)
(\tilde z_i)$, which is stronger
than $v_i\in T^{[\varepsilon_i]}(\tilde z_i)$.
In such a case, Item (a) would guarantee that
$v_j\in \left(S+N_{\mathcal{X}}^{[\varepsilon_j]}\right)
(\tilde z_j)$. Unfortunately, the observation is
not true for the Item (b).

The next result was proved in \cite[Corollary 1]{error-bound}
\begin{lemma}
  \label{lm:d}
  If
 $\bpt{z}\in{\R^p}$, $\lambda>0$, and
  $v\in T^{[\varepsilon]}(z)$, then
  \begin{align*}
    \norm{\lambda v+z-\bpt{z}}^2+2\lambda\varepsilon
    \geq \Norm{z-(\lambda T+I)^{-1}\bpt{z}}^2
    +\Norm{\lambda v-\left(\bpt{z}-(\lambda T+I)^{-1}\bpt{z}\right)}^2.
  \end{align*}
\end{lemma}

\section{The smooth convex programming problem}
\label{sec:sqp2}

Consider the smooth convex optimization problem
\eqref{eq:np.i}, i.e.,
\begin{align}
  \label{eq:np}
(P)\qquad
    \min& f(x)\qquad
    \mathrm{s.t.}\;\; g(x)\leq 0,
\end{align}
where $f: \R^n \to \R$ and $g=(g_1,\ldots,g_m): \R^n \to \R^m$.
From now on we assume that:

\begin{itemize}
\item[O.1)] $f,g_1,\dots,g_m$ are convex $\mathscr{C}^2$ functions;
\item[O.2)] the Hessians of $f$ and $g_1,\ldots,g_m$ are Lipschitz
  continuous with Lipschitz constants $L_0$ and $L_1,\ldots,L_m$,
  respectively, with $L_i\neq 0$ for some $i\geq 1$;
\item[O.3)] there exists $(x,y)\in\R^n\times\R^m$ satisfying
   Karush-Kuhn-Tucker conditions for \eqref{eq:np},
 \begin{align}
   \label{eq:kkt.p}
   \nabla f(x)+ \nabla g(x)y=0,
\;\;
   g(x)\leq 0,\;\;\; y\geq 0,\;\;
   \inner{y}{g(x)}=0.   
 \end{align}
\end{itemize}


The \emph{canonical Lagrangian} of problem~\eqref{eq:np}
$\Lb:\R^n\times \R^m \to \R$ and the corresponding \emph{saddle-point
  operator} $\So:\R^n\times\R^m\to\R^n\times\R^m$ are, respectively,
\begin{align}
  \label{def:lag.e}
    \Lb(x,y)&:=f(x)+\inner{y}{g(x)},\quad
    \So(x,y):=
              \begin{bmatrix}
                \nabla_x\Lb(x,y)\\ -\nabla_y\Lb(x,y)
              \end{bmatrix}
    = \begin{bmatrix}
      \nabla f(x)  + \nabla g (x) y \\
      -g(x)
    \end{bmatrix}.
\end{align}
The
\emph{normal cone operator} of $\R^n\times\R^m_+$,
$\Nc_{\R^n\times\R^m_+}:\R^n\times\R^m
\tos\R^n\times\R^m$, is the
subdifferential of the indicator function of this set
$\delta_{\R^n\times\R^m_+}:\R^n\times\R^m
\to\overline{\R}$, that is,
\begin{align}
  \label{eq:if.nc}
  \delta_{\R^n\times\R^m_+}(x,y):=
  \begin{cases}
    0,& \text{if }y\geq 0;\\
    \infty& \text{otherwise},
  \end{cases}\qquad
  \Nc_{\R^n\times\R^m_+}:=\partial   \delta_{\R^n\times\R^m_+}.
\end{align}

Next we review some reformulations of \eqref{eq:kkt.p}.

\begin{proposition}
  \label{pr:rf.sdp}
  The point-to-set operator $\So+\Nc_{\R^n\times\R^m_+}$ is maximal
  monotone and for any $(x,y)\in\R^n\times\R^m$ the following conditions
  are equivalent:
  \begin{enumerate}
  \item[\emph{(a)}] $ \nabla f(x)+\nabla g(x)y=0$,
 $g(x)\leq 0$, $y\geq 0$,
  and $\inner{y}{g(x)}=0$;
  \item[\emph{(b)}] $(x,y)$ is a solution of the saddle-point problem
    $\max_{y\in\R^m_+}\min_{x\in\R^n}\; f(x)+\inner{y}{g(x)}$;
  \item[\emph{(c)}] $(x,y)$ is a solution of the complementarity problem
    \begin{align*}
      (x,y)\in \R^n\times\R^m;\;w\in\R^m;\;\;
      \So(x,y)-(0,w)=0;\;\;
      y,w\geq 0;\;\;\inner{y}{w}=0;
    \end{align*}
    \item[\emph{(d)}] 
  $(x,y)$ is a solution of the monotone inclusion
 problem
   $0\in\left(\So+\Nc_{\R^n\times\R^m_+}\right)(x,y)$.
  \end{enumerate}
\end{proposition}

Next we review some reformulations of 
the saddle-point problem in
\eqref{eq:pmm}.

\begin{proposition}
  \label{pr:rf.rsdp}
  Take $(\bpt{x},\bpt{y})\in \R^n\times\R^m$ and $\lambda>0$.
  The following conditions
  are equivalent:
  \begin{enumerate}
  \item[\emph{(a)}] \label{it:rsdp}
    $(x,y)$ is the solution of the
    regularized saddle-point problem\\
$\mbox{}\qquad\qquad\min_{x\in\R^n}\max_{y\in\R^m_+}\;
 f(x)+\inner{y}{g(x)}
+\dfrac{1}{2\lambda}
      \left(\norm{x-\bpt{x}}^2-\norm{y-\bpt{y}}^2\right)$;
  \item[\emph{(b)}] \label{it:rmcp}
    $(x,y)$ is the solution of the regularized 
complementarity problem
    \begin{align*}
      (x,y)\in \R^n\times\R^m;\;\;
      \lambda \bigg(\So(x,y)-(0,w)\bigg)+(x,y)-(\bpt{x},\bpt{y})=0;\;\;
      y,w\geq 0;\;\;\inner{y}{w}=0;
    \end{align*}
  \item[\emph{(c)}] \label{it:psp.2}
    $(x,y)=\left(\lambda (\So+\Nc_{\R^n\times\R^m_+})
  +\emph{I}\right)^{-1}
    (\bpt{x},\bpt{y})$.
  \end{enumerate}
\end{proposition}

It follows from Propositions~\ref{pr:rf.sdp} 
and \ref{pr:rf.rsdp} that
\eqref{eq:kkt.p} is equivalent to 
the monotone inclusion problem
\[
0\in(\So+\Nc_{\R^n\times\R^m_+})(x,y)
\]
and that \eqref{eq:pmm} is the PPM iteration for this
inclusion problem.
Therefore, the convergence analysis
of the Rockafellar's PMM follows
from Rockafellar's convergence analysis of the PPM.

\section{An error measure for regularized saddle-point problems}
\label{sec:em.pp}

We will present a modification of Rockafellar's PMM which uses
approximate solutions of the regularized saddle-point problem
\eqref{eq:pmm} satisfying a relative error tolerance.
To that effect, in this section we define a generic instance of
problem~\eqref{eq:pmm},
define an error measure for approximate solutions of this
generic instance, and analyze some properties of the proposed error
measure.

Consider, for $\lambda> 0$ and $(\bpt{x},\bpt{y}) \in \R^n\times\R^m$, a generic
instance of the regularized saddle-point problem to be solved in each iteration
of Rockafellar's PMM,
\begin{align}
  \label{eq:g.rsdp}
  & \min_{x\in\R^n}\max_{y\in\R^m_+}
    \;f(x)+\inner{y}{g(x)}
   +\dfrac{1}{2\lambda}   (\norm{x-\bpt{x}}^2-\norm{y-\bpt{y}}^2).
\end{align}
Define
 for  $\lambda\in\R$ and $(\bpt{x},\bpt{y})\in\R^n\times\R^m$
\begin{align}
  \label{eq:Psi.1}
  \begin{split}
    &\Psi_{\So,(\bpt{x},\bpt{y}),\lambda}:\R^n\times\R^m_+\to\R,\\
    &\Psi_{\So,(\bpt{x},\bpt{y}),\lambda}(x,y)
    \begin{aligned}[t]
      &:= \min_{w\in\R^m_+}
  \Norm{\lambda
        \bigg(\So(x,y)-(0,w)\bigg)
        +(x,y)-(\bpt{x},\bpt{y})}^2
      +2\lambda\inner{y}{w}.
    \end{aligned}
  \end{split}
\end{align}
For $\lambda>0$, this function is 
trivially an error measure for the
complementarity problem on 
Proposition~\ref{pr:rf.rsdp} (b), 
a problem which is equivalent to~\eqref{eq:g.rsdp},
by Proposition \ref{pr:rf.rsdp} (a);
hence, $\Psi_{\So,(\bpt{x},\bpt{y}),\lambda}(x,y)$ is an error measure for
\eqref{eq:g.rsdp}.

In the context of complementarity problems, the quantity $\inner{y}{w}$ in
\eqref{eq:Psi.1} is refered to as the \emph{complementarity gap}. 
Next we show
that the complementarity gap is related to the 
$\varepsilon$-subdifferential of $\delta_{\R^n\times\R^m_+}$ and 
to the $\varepsilon $-enlargement of the normal cone
operator of $\R^n\times\R^m_+$.
Direct use of \eqref{eq:teps} and
of the definition of the $\varepsilon$-subdifferential~\cite{MR0178103} yields
%
\begin{align}
  \label{eq:npe}
  \begin{aligned}
    &\forall (x,y)\in \R^n\times\R^m_+,\;\varepsilon\geq 0\\
    &
\partial_\varepsilon\delta_{\R^n\times\R^m_+}(x,y)=
    \Nc_{\R^n\times\R^m_+}^{[\varepsilon]}(x,y)
    =
    \Set{-(0,w)\;|\; w\in\R^m_+,\; \inner{y}{w}\leq \varepsilon}.\qquad
  \end{aligned}
\end{align}

Since
\begin{align}
  \label{eq:tildew}
  \argmin_{w\in\R^m_+}
  \norm{\lambda
        \big(\So(x,y)-(0,w)\big)
        +(x,y)-(\bpt{x},\bpt{y})}^2
      +2\lambda\inner{y}{w}
 = (g(x)+\lambda^{-1}\bpt{y})_-,
\end{align}
it follows from definition \eqref{eq:Psi.1} that
\begin{align}
  \label{eq:Psi.2}
  \begin{split}
    &\Psi_{\So,(\bpt{x},\bpt{y}),\lambda}(x,y)
    \begin{aligned}[t]
      &=\norm{\lambda (\nabla f(x)+\nabla g(x)y)
        +x-\bpt{x}}^2 +\norm{y-(\lambda g(x)+\bpt{y})_+}^2\\
      &+2\inner{y}{ (\lambda g(x)+\bpt{y})_-}\\
      &=\Norm{\lambda\So(x,y)+(x,y)-(\bpt{x},\bpt{y})}^2
      -\norm{(\lambda g(x)+\bpt{y})_-}^2\,
    \end{aligned}
  \end{split}
\end{align}
for any $(x,y)\in\R^n\times\R^m_+$.
%
%

\begin{lemma}
  \label{lm:loc3}
  If $\lambda>0$, $\bpt{z}=(\bpt{x},\bpt{y})\in\R^n\times\R^m$, 
  $z=(x,y)\in\R^m\times\R^m_+$ and $w,v,\varepsilon$ are 
  defined as
  \[
  w:=(g(x)+\lambda^{-1}\bpt{y})_-,\qquad
  v:=\So(z)-(0,w),\qquad
  \varepsilon:=\inner{y}{w},
  \]
  then
  \begin{enumerate}
  \item[\emph{(a)}]
    \label{it:psi.1}
    $-(0,w)\in \partial_{\varepsilon} \delta_{\R^n\times\R^m_+} (z)=
    \Nc_{\R^n\times\R^m_+}^{[\varepsilon]}(z)$;
  \item[\emph{(b)}]
    \label{it:psi.2}
     $v\in(\So+\Nc_{\R^n\times\R^m_+}^{[\varepsilon]})(z)\subset
    (\So+\Nc_{\R^n\times\R^m_+})^{[\varepsilon]}(z)$,
    $\norm{\lambda v+z-\bpt{z}}^2+2\lambda\varepsilon=
    \Psi_{\So,\bpt{z},\lambda}(z)$;
  \item[\emph{(c)}]
    \label{it:psi.4}
    $\Norm{z-\big(\lambda(\So+\Nc_{\R^n\times\R^m_+})+I\big)^{-1}(\bpt{z})}
  \leq\sqrt{\Psi_{\So,\bpt{z},\lambda}(z)}$.
  \end{enumerate}
\end{lemma}

\begin{proof}
  Item (a) follows trivially from the 
  definitions of $w$ and
  $\varepsilon$, and \eqref{eq:npe}. 
 The first inclusion in
  item (b) follows from
 the definition of $v$ and item (a); the second
  inclusion follows from
 direct calculations and \eqref{eq:teps}; the 
 identity in item (b)
 follows from the definitions of $w$ and $\varepsilon$,
  \eqref{eq:Psi.1} and \eqref{eq:tildew}. 
 Finally, item (c) follows
  from item (b) and Lemma~\ref{lm:d}.
\end{proof}

Now we will show how to 
update $\lambda$ so that
$\Psi_{\So,\bpt{z},\lambda}(x,y)$ does not 
increase when $\bpt{z}$ is
updated like $z_{k-1}$ is updated to $z_{k}$ in \eqref{eq:rhpe}.

\begin{proposition}
  \label{pr:relax2}
  Suppose that $\lambda >0$,
  $\bpt{z}=(\bpt{x},\bpt{y})\in\R^n\times\R^m$,
  ${z}=({x},{y})\in\R^n\times\R^m_+$ and
  define
  \begin{align*}
    w&:=(g({x})+\lambda^{-1}\bpt{y})_-,\qquad 
       v:=\So({z})-(0,w),\qquad
   \bpt{z}(\tau):=\bpt{z}-\tau \lambda v=
 :(\bpt{x}(\tau),\bpt{y}(\tau)).
  \end{align*}
  For any $\tau\in[0,1]$,
  \begin{align*}
    &\bpt{x}(\tau)=\bpt{x}-\tau\lambda(\nabla f({x})+\nabla g({x}){y}),\;\;
    \bpt{y}(\tau)=  \bpt{y}+\tau\lambda( g({x})+
    ( g({x})+\lambda^{-1}\bpt{y})_-)\,,\;\;\\
    &\Psi_{\So,\bpt{z}(\tau),(1-\tau)\lambda}({z})
    \leq \Psi_{\So,\bpt{z},\lambda}({z}).
  \end{align*}
  If, additionally, $\bpt{y}\geq 0$ then, 
 for any $\tau\in[0,1]$,
  $\bpt{y}(\tau)\geq 0$.
\end{proposition}
\begin{proof}
  Direct use of the definitions of $w$, $v$ and
  $\bpt{z}(\tau)$ yields the expressions 
  for $\bpt{x}(\tau)$ and $\bpt{y}(\tau)$
  as well as the identity
 \begin{align*}
    (1-\tau)\lambda
        \left(\So({z})-(0,w)\right)
    +{z}-\bpt{z}(\tau)=
    \lambda
        \left(\So({z})-(0,w)\right)
        +{z}-\bpt{z},
  \end{align*}
 which, in turn, combined with \eqref{eq:Psi.1}
 gives, for any
 $\tau \in[0,1]$,
  \begin{align*}
    \Psi_{\So,\bpt{z}(\tau),(1-\tau)\lambda}({z})
    & \leq   \norm{(1-\tau)\lambda
        \left(\So({z})-(0,w)\right)
        +{z}-\bpt{z}(\tau)}^2
      +2(1-\tau)\lambda\inner{{y}}{w}
\\
    & =   \norm{\lambda
        \left(\So({z})-(0,w)\right)
        +{z}-\bpt{z}}^2
      +2(1-\tau)\lambda\inner{{y}}{w}\\
    &\leq \Psi_{\So,\bpt{z},\lambda}({z}),
  \end{align*}
  where the second inequality follows 
 from \eqref{eq:Psi.1}, \eqref{eq:tildew}, 
 the assumption $\tau \in[0,1]$ and the definition of
  $ w$. 
  To prove the second part of the proposition, 
  observe that, for any
  $\tau\in[0,1]$, $\bpt{y}(\tau)$ is a convex combination of $\bpt{y}$ and
  $\bpt{y}(1)=(\lambda g({x})+\bpt{y})_+$.
\end{proof}

The next lemma and the next proposition
provide quantitative and qualitative
estimations of the dependence of
$\Psi_{\So,(\bpt{x},\bpt{y}),\lambda}(x,y)$ on $\lambda$.

\begin{lemma}
  \label{lm:ef}
  If $\psi(\lambda):=\Psi_{\So,\bpt{z},\lambda}(z)$ for $\lambda\in\R$, where
 $\bpt{z}=(\bpt{x},\bpt{y})\in\R^n\times\R^m$,
  and $z=(x,y) \in \R^n \times \R^m_+$,
   then
  \begin{enumerate}
  \item[\emph{(a)}] $\psi$ is
    convex, differentiable and piecewise quadratic;
  \item[\emph{(b)}]
    $\dfrac{d}{d\lambda}\psi(\lambda)
    =2\left(\Inner{\So(z)} {\lambda(\So(z)-(0,w))+z-\bpt{z}}\right)$
    where $w=(g(x)+\lambda^{-1}\bpt{y})_-$\;;
  \item[\emph{(c)}]
    $\psi(\lambda) \leq (\norm{\lambda \So(z)}
      +\norm{z-\bpt{z}})^2$;
    \item[\emph{(d)}]
  $\lim_{\lambda\to\infty}\psi(\lambda)<\infty$ if and only
      if $(x,y)$ is a solution of \eqref{eq:kkt.p}.
    \end{enumerate}
  \end{lemma}

\begin{proof}
  The proof follows trivially from \eqref{eq:Psi.2}.
\end{proof}


\begin{proposition}
  \label{pr:linc}
  If $\bpt{z}=(\bpt{x},\bpt{y})\in\R^n\times\R^m$, 
  $z=({x},{y})\in \R^n\times\R^m_+$
  and $0<\mu\leq\lambda$ then
  \begin{align*}
    \sqrt{\Psi_{\So,\bpt{z},\lambda}(z)}
    \leq \dfrac{\lambda}{\mu} 
    \sqrt{\Psi_{\So,\bpt{z},\mu}(z)}
    +\dfrac{\lambda-\mu}{\mu}
  \norm{z-\bpt{z}}.
  \end{align*}
\end{proposition}

\begin{proof}
  Let $w:=(g({x})+\mu^{-1}\bpt{y})_-$ and
  \begin{align*}
    r_\mu:=\mu\bigg(\So({z})-(0,w)\bigg)+{z}-\bpt{z},\;\;
    r_\lambda:=\lambda\bigg(\So({z})-(0,w)\bigg)+{z}-\bpt{z}.
  \end{align*}
  It follows from the latter definitions, 
 \eqref{eq:Psi.1} and
  \eqref{eq:tildew} that
  $\Psi_{\So,\bpt{z},\mu}({z})=\norm{r_\mu}^2+
  2\mu\inner{{y}}{w}$ and
  \begin{align*}
    \Psi_{\So,\bpt{z},\lambda}({z})
 \leq \norm{r_\lambda}^2+2\lambda\inner{{y}}{w}
    &=\Norm{\dfrac{\lambda}{\mu}r_\mu+\dfrac{\mu-\lambda}{\mu}
      ({z}-\bpt{z})}^2
      +\dfrac{\lambda}{\mu}2\mu\inner{{y}}{w}\\
    & \leq \left(\dfrac{\lambda}{\mu}\right)^2
    \left(\norm{r_\mu}^2+
      2\mu\inner{{y}}{w}\right)+
      2\dfrac{\lambda}{\mu}\dfrac{\lambda-\mu}{\mu}
      \norm{r_\mu}\norm{{z}-\bpt{z}}\\
			&\;\;\;\;+\left(\dfrac{\lambda-\mu}{\mu}\norm{{z}-\bpt{z}}\right)^2\\
    & \leq \left(\dfrac{\lambda}{\mu}\right)^2
      \Psi_{\So,\bpt{z},\mu}({z})+2\dfrac{\lambda}{\mu}
      \sqrt{\Psi_{\So,\bpt{z},\mu}({z})}\dfrac{\lambda-\mu}{\mu}\norm{{z}-\bpt{z}}\\
      &\;\;\;\;+\left(\dfrac{\lambda-\mu}{\mu}\norm{{z}-\bpt{z}}
      \right)^2, 
  \end{align*}
  where the first inequality follows from the assumption
  $0<\mu\leq\lambda$. 
  The conclusion follows trivially from the
  latter inequality.
\end{proof}

\section{Quadratic approximations of the smooth convex
  programming problem}
\label{sec:qa}

In this section we use second-order 
approximations of $f$ and $g$ around a point
$\tilde x\in\R^n$ to define a second-order 
approximation of
problem~\eqref{eq:np} around such a point.
We also define a local model of~\eqref{eq:pmm}, 
where second-order
approximations of $f$ and $g$ around $\tilde x$ substitute these functions,
and give conditions on a point 
$(\tilde x,\tilde y)$ under which a solution of
the local model is a better approximation to the solution of~\eqref{eq:pmm}
than this point.

For $\tilde x\in \R^n$, let $f_{[\tilde x]}$ and
$g_{[\tilde x]} = (g_{1,[\tilde x]}, \ldots ,g_{m,[\tilde x]})$
be the quadratic approximations of $f$ and $g=(g_1,\ldots,g_m)$ around
$\tilde x$, that is,
\begin{align}
  \label{eq:qpfg}
  \begin{split}
    f_{[\tilde x]}(x) &:= f(\tilde x)+
    \nabla f(\tilde x)^T(x-\tilde x)
    +\dfrac12(x-\tilde x)^T\nabla^2f(\tilde x)(x-\tilde x)\\
    g_{i,[\tilde x]}(x) &:=g_i(\tilde x)+\nabla
    g_i(\tilde x)^T(x-\tilde x)
    +\dfrac12(x-\tilde x)^T\nabla^2g_i(\tilde x)
   (x-\tilde x),\quad 
    i=1,\ldots,m.
    \end{split}
\end{align}
We define 
\begin{align}
  \label{eq:np2.xi}
  (P_{[\,\tilde x\,]})\qquad
  \min& f_{[\tilde x]}(x)\qquad
  \mathrm{s.t.}\;\; g_{[\tilde x]}(x)\leq 0
\end{align}
as the quadratic approximation of problem~\eqref{eq:np}
around $\tilde x$.
%
The canonical Lagrangian of \eqref{eq:np2.xi},
$\Lb_{[\tilde x]}:\R^n\times\R^m\to\R$, and the corresponding
saddle-point operator,
$\So_{[\tilde x]}:\R^n\times\R^m\to\R^n\times\R^m$, are, respectively,
\begin{align}
  \label{def:lag.a}
  \begin{split}
    \Lb_{[\tilde x]}(x,y)&:=f_{[\tilde x]}(x)+\inner{y}{g_{[\tilde x]}(x)},
    \\
    \So_{[{\tilde{x}}]}(x,y)&:=
    \begin{bmatrix}
            \nabla_{x}\Lb_{[\tilde x]}(x,y)\\
      -\nabla_{y}\Lb_{[\tilde x]}(x,y)
    \end{bmatrix}
    = \begin{bmatrix}
      \nabla f_{[\tilde x]}(x)  + \nabla g_{[\tilde x]}(x)y \\
      -g_{[\tilde x]}(x)
    \end{bmatrix}.
  \end{split}
\end{align}
Since  $\Lb_{[\tilde x]}(x,y)$ is a 3rd-degree polynomial in $(x,y)$ and
the components of $\So_{[\tilde x]}(x,y)$ are 2nd-degree
polynomials in $(x,y)$,
neither $\Lb_{[\tilde x]}$  is a
quadratic  approximation of $\Lb$ nor $\So_{[\tilde x]}$ 
is a linear approximation of 
$\So$;
nevertheless, this 3-rd degree functional and that componentwise
quadratic operator are, respectively, the canonical Lagrangian and the
associated saddle-point operator of a quadratic approximation of $(P)$
around $\tilde x$, namely, $(P_{[\tilde x]})$.
So, we may say that $\Lb_{[\tilde x]}$ and $\So_{[\tilde x]}$ are
approximations of
$\Lb$ and $\So$ \emph{based on quadratic approximations} of $f$ and $g$.

Each iteration of Rockafellar's PMM 
\emph{applied to problem}
$(P_{[\tilde x]})$ requires the solution of an instance of the generic
regularized saddle-point problem
\begin{align}
  \label{eq:rsdp.q}
  \begin{split}
    &\min_{x\in\R^n}\max_{y\in\R^m_+}\;f_{[\tilde
      x]}(x)+\Inner{y}{g_{[\tilde x]}(x)}
    +\dfrac{1}{2\lambda} (\norm{x-\bpt{x}}^2-
\norm{y-\bpt{y}}^2),
  \end{split}
\end{align}
where $\lambda>0$ and 
$(\bpt{x},\bpt{y})\in\R^n\times \R^m$. 
It follows from Proposition~\ref{pr:rf.rsdp} that this problem is
equivalent to the complementarity problem
\begin{align*}
  \begin{split}
   &(x,y)\in\R^n\times\R^m_+;\;\;\;
    \lambda\So_{[\tilde x]}(x,y)+(x,y)-(\bpt{x},\bpt{y})=(0,w);\;\;\;
   y,w\geq0;\;\;\; \inner{y}{w}=0.
  \end{split}
\end{align*}
To analyze the error of 
substituting $\So$ by $\So_{[\tilde x]}$
we introduce the
notation:
\begin{align}
  L_g=(L_1,\dots,L_m);\quad
  |(y_1,\dots,y_m)|=(|y_1|,\dots,|y_m|),\;\;
  (y_1,\dots,y_m)\in\R^m.
\end{align}

\begin{lemma}
  \label{lm:app.1}
  For any $({x},{y})\in\R^n\times\R^m$ and $\tilde x\in\R^n$
  \begin{align*}
    \Norm{\So({x},{y})-\So_{[\tilde x]}({x},{y})}
    & \leq \dfrac{L_0+\inner{L_g}{\abs{{y}}}}{2}\norm{{x}-\tilde x}^2
      + \dfrac{\norm{L_g}}{6}\norm{{x}-\tilde x}^3.
  \end{align*}
\end{lemma}

\begin{proof}
  It follows from triangle inequality, \eqref{eq:qpfg}
  and assumption
  (O.2) that
  \begin{align*}
\norm{\nabla_x\mathscr{L}_{[\tilde x]}(x,y)-\nabla_x\mathscr{L}(x,y)} 
& \leq
      \norm{\nabla f_{[\tilde x]}({x})-\nabla f({x})}+
      \norm{(\nabla g_{[\tilde x]}({x})-\nabla g({x})){y}}
\\& 
 \leq \left(\dfrac{L_0}{2}+\sum_{i=1}^m\abs{{y}_i}\dfrac{L_i}{2}\right)
      \norm{{x}-\tilde x}^2
  \end{align*}
  and
  \begin{align*}
    \norm{g_{[\tilde x]}({x})-g({x})}
     = \sqrt{\sum_{i=1}^m\left(g_{i,[\tilde x]}({x})-g_i({x})\right)^2}
      &\leq
      \sqrt{\sum_{i=1}^m\left(\dfrac{L_i}{6}\norm{{x}-\tilde x}^3\right)^2}\\
      &=\dfrac{\norm{L_g}}{6}\norm{{x}-\tilde x}^3.
  \end{align*}
  To end the proof, use the above inequalities, 
  \eqref{def:lag.e} and
  \eqref{def:lag.a}.
\end{proof}

Define, for $(\bpt{x},\bpt{y})\in\R^n\times\R^m$ and $\lambda > 0$,
\begin{align}
  \label{eq:viz}
  \mathcal{N}_\theta((\bpt{x},\bpt{y}),\lambda)
  &:= \Set{(x,y)\in\R^n\times\R^m_+\left|
        \begin{aligned}
    &\lambda\left(
          \dfrac{L_0+\inner{L_g}{\abs{y}}}{2}+\dfrac{2\norm{L_g}}{3}
       \rho
        \right)\rho\leq
\theta,\\
&\text{ where }\rho=
\sqrt{ \Psi_{\So,(\bpt{x},\bpt{y}),\lambda}(x,y)}
        \end{aligned}
      }.\right.
\end{align}
The next proposition shows that, if
$(\tilde x,\tilde y) \in 
\mathcal{N}_\theta((\bpt{x},\bpt{y}),\lambda)$ with
$\theta \leq 1/4$, then the 
solution of the regularized saddle-point problem
\eqref{eq:rsdp.q} is a better 
approximation than $(\tilde x,\tilde y)$ to
the solution of the regularized saddle-point 
problem~\eqref{eq:g.rsdp},
with respect to the merit 
function $\Psi_{\So,(x,y),\lambda}$.

\begin{proposition}
  \label{pr:rc}
  If $\lambda>0$, $(\bpt{x},\bpt{y})\in\R^n\times\R^m$,
  $(\tilde x,\tilde y)\in\R^n\times\R^m_+$ and 
  \begin{align*}
      (x,y)=&\arg\;\min_{x\in\R^n}\max_{y\in\R^m_+}\;f_{[\tilde
        x]}(x)+\Inner{y}{g_{[\tilde x]}(x)}
      +\dfrac{1}{2\lambda} 
   \left(\norm{x-\bpt{x}}^2-\norm{y-\bpt{y}}^2\right),
  \end{align*}
  then
  \begin{align*}
    \norm{(\tilde x,\tilde y)-({x},{y})} \leq
     \sqrt{\Psi_{\So,(\bpt{x},\bpt{y}),\lambda}(\tilde x,\tilde y)}.
  \end{align*}
  If, additionally,
  $(\tilde x,\tilde y)\in
  \mathcal{N}_\theta((\bpt{x},\bpt{y}),\lambda)$
  with $0\leq \theta\leq1/4$, then
  $$    \sqrt{\Psi_{\So,(\bpt{x},\bpt{y}),\lambda}({ x},{ y})}
  \leq \theta    \sqrt{\Psi_{\So,(\bpt{x},\bpt{y})\lambda}(\tilde x,\tilde y)}
  $$
  and $({ x},{ y})\in
  \mathcal{N}_{\theta^2}((\bpt{x},\bpt{y}),\lambda)$.
\end{proposition}

\begin{proof}
  Applying Lemma~\ref{lm:loc3} to \eqref{eq:rsdp.q} and using
  \eqref{def:lag.a} we conclude that
  \begin{align*}
    \norm{(\tilde x,\tilde y)-(x,y)} \leq
    \sqrt{\Psi_{\So_{[\tilde x]},(\bpt{x},\bpt{y}),\lambda}(\tilde x,\tilde y)}.   
  \end{align*}
  It follows from \eqref{eq:qpfg}, \eqref{def:lag.a}, and
  \eqref{def:lag.e} that
  $\So_{[\tilde x]}(\tilde x,\tilde y)=\So(\tilde x,\tilde y)$, which,
  combined with \eqref{eq:Psi.1}, implies that
  $\Psi_{\So_{[\tilde x]},(\bpt{x},\bpt{y}),\lambda}(\tilde x,\tilde
  y)=\Psi_{\So(\bpt{x},\bpt{y}),\lambda}(\tilde x,\tilde y)$. 
	To prove the first part of the proposition, combine this result with
  the above inequality.

  To simplify the proof of the second part of the proposition, define
  \begin{align*}
      &\tilde \rho
      =\sqrt{\Psi_{\So_{[\tilde x]},(\bpt{x},\bpt{y}),\lambda}(\tilde x,\tilde y)},\
   \qquad
    w
    =( g( x)+\lambda^{-1}\bpt{y})_-,\\ 
		&r=\lambda\bigg(\So( x, y)-(0,w)\bigg)+( x, y)-(\bpt{x},\bpt{y}).
  \end{align*}
   Since $( x, y)$ is the solution of \eqref{eq:rsdp.q},
   \begin{align*}
     \lambda\bigg(\So_{[\tilde x]}( x, y)-(0,w)\bigg)+( x, y)
     -(\bpt{x},\bpt{y})=0,\quad  y,w\geq0,\quad
     \inner{ y}{w}=0.
   \end{align*}
   Therefore, 
  $r=\lambda(\So( x, y)-\So_{[\tilde x]}( x, y))$. 
  Using also~\eqref{eq:Psi.1}, 
  Lemma~\ref{lm:app.1} and the first part of the
   proposition we conclude that
   \begin{align*}
     \sqrt{\Psi_{\So,(\bpt{x},\bpt{y}),\lambda}( x, y)}
     \leq \sqrt{\norm{r}^2+2\lambda\inner{ y}{w}}
       &=\lambda\Norm{\So( x, y)
       -\So_{[\tilde x]}( x, y)}\\
       &\leq \lambda\left(\dfrac{L_0+\inner{L_g}{|y|}}{2}
       +\dfrac{\norm{L_g}}{6}\tilde \rho\right) 
   \tilde \rho^2.
   \end{align*}
   Moreover, it 
   follows from the Cauchy-Schwarz inequality,
  the first part
  of the proposition
  and the definition of $\tilde \rho$
  that
   \begin{align*}
     \inner{L_g}{|y|} \leq
     \inner{L_g}{|\tilde y|}+
     \norm{L_g}\norm{y-\tilde y}
     \leq \inner{L_g}{|\tilde y|}+\norm{L_g}\tilde \rho.
   \end{align*}
 Therefore
   \begin{align*}
     \sqrt{\Psi_{\So,(\bpt{x},\bpt{y}),\lambda} ( x, y) }
     \leq \lambda\left(\dfrac{L_0+
    \inner{L_g}{|\tilde y|}}{2}
     +\dfrac{2\norm{L_g}}{3}\tilde 
   \rho\right)\tilde \rho^2.
   \end{align*}

   Suppose that
   $(\tilde x,\tilde y)\in
   \mathcal{N}_\theta((\bpt{x},\bpt{y}),\lambda)$
   with $0\leq \theta\leq1/4$. It follows trivially from this
   assumption, \eqref{eq:viz}, the definition of $\tilde \rho$, and the
   above inequality, that the inequality in the second part of the
   proposition holds.
   To end the proof of the second part, let
   $\rho= \sqrt{\Psi_{\So_{[\tilde
         x]},(\bpt{x},\bpt{y}),\lambda}( x,  y)}$.
   Since $\rho \leq \theta \tilde \rho\leq \tilde 
   \rho/4$ and
   $\inner{L_g}{|y|}\leq 
   \inner{L_g}{|\tilde y|}+\norm{L_g}\tilde \rho$,

   \begin{align*}
     \lambda\left(
     \dfrac{L_0+\inner{L_g}{|y|}}{2}+
    \dfrac{2\norm{L_g}}{3}
     \rho
     \right)\rho
     & \leq
       \lambda\left(
       \dfrac{L_0+\inner{L_g}{|\tilde y|}+
       \norm{L_g}\tilde \rho}{2}
       +\dfrac{2\norm{L_g}}{3}
       \dfrac{\tilde \rho}{4}
       \right)\theta \rho\\
     & =
       \lambda\left(
       \dfrac{L_0+\inner{L_g}{|\tilde y|}}{2}
       +\dfrac{2\norm{L_g}}{3}\tilde \rho
       \right)\theta\tilde \rho\leq \theta^2,
   \end{align*}
   where the last inequality follows from the assumption
   $(\tilde x,\tilde y)\in \mathcal{N}_\theta((\bpt{x},\bpt{y}),\lambda)$ and
   \eqref{eq:viz}.
   To end the proof use the definition of $\rho$, the above inequality
   and \eqref{eq:viz}.     
\end{proof}

In view of the preceding proposition, for a given
$(\bpt{x},\bpt{y})\in\R^n\times\R^m$ and $\theta>0$, it is natural to
search for $\lambda>0$ and
$(x,y)\in\mathcal{N}_\theta((\bpt{x},\bpt{y}),\lambda)$.

\begin{proposition}
  \label{pr:boot}
  For any $(\bpt{x},\bpt{y}) \in \R^n\times\R^m$,
  $(x,y) \in \R^n\times\R^m_+$, and $\theta>0$ there
  exists $\bar \lambda > 0$ such that
  $(x,y) \in \mathcal{N}_\theta((\bpt{x},\bpt{y}),\lambda)$ for any
  $\lambda \in (0, \bar\lambda]$.
\end{proposition}

\begin{proof}
The proof follows from the definition \eqref{eq:viz} 
and Lemma~\ref{lm:ef}(c).
\end{proof}

The neighborhoods $\mathcal{N}_\theta$ as well as the next defined function
will be instrumental in the definition and analysis of Algorithm 1, to be
presented in the next section.

\begin{definition}\label{df:rho}
   For
   $\alpha>0$ and $y\in\R^m$,
 $\rho(y,\alpha)$ stands for the largest root of 
 \begin{align*}
   \left( \dfrac{L_0+\inner{L_g}{\abs{y}}}{2}+
     \dfrac{2\norm{L_g}}{3}\rho\right)\rho=\alpha.
 \end{align*}
\end{definition}

Observe that 
 for any $\lambda,\theta>0$ and
$(\bpt{x},\bpt{y})\in \R^n\times\R^m$,
\begin{align}
  \label{eq:viz.2}
  \mathcal{N}_\theta((\bpt{x},\bpt{y}),\lambda)=\Set{ (x,y)\in\R^n\times
    \R^m_+\;\left|
      \sqrt{\Psi_{\So,(\bpt{x},\bpt{y}),\lambda}(x,y)}
      \leq \rho\left(y,{\theta}/{\lambda}\right)}
\right. .
\end{align}
Moreover, since $\rho(y,\alpha)$ is the largest root 
of a quadratic it follows that it has an
explicit formula.

\section{A relaxed hybrid proximal extragradient
  method of multipliers based on second-order approximations}
\label{sec:sqr-hpemm}

In this section we consider the 
smooth convex programming problem \eqref{eq:np}
where assumptions O.1, O.2 and O.3 are assumed
to hold. Aiming at finding approximate solutions
of the latter problem, we propose a new method,
called (relaxed) hybrid proximal extragradient method
of multipliers based on quadratic 
approximations (hereafter
rHPEMM-2o), which is a modification of 
Rockafellar's PMM in
the following senses:
in each iteration either a relaxed extragradient step is executed or
a second order approximation of \eqref{eq:g.rsdp} is solved. More specifically,
each iteration $k$ uses the (available) variables
\[
(x_{k-1},y_{k-1}),\;\;\;\;(\tilde x_{k},\tilde y_{k})\in\R^n\times\R^m_+,\;
\text{ and }\lambda_k>0
\]
to generate
\[
(x_k,y_k),\;\;\;\;(\tilde x_{k+1},\tilde y_{k+1})\in\R^n\times\R^m_+,\;
\text{ and }\lambda_{k+1}>0
\]
in one of two ways.
Either
\\
(1) $(x_k,y_k)$ is obtained from $(x_{k-1},y_{k-1})$ via a relaxed
extragradient step, 
\[(x_k,y_k)=(x_{k-1},y_{k-1})-\tau\lambda_k v_k,\;\;
v_k\in (\So+\Nc_{\R^n\times\R^m_+})^{\varepsilon_k}(\tilde x_k,\tilde y_k)
\]
in which case $(\tilde x_{k+1},\tilde y_{k+1})=(\tilde x_{k},\tilde y_{k})$ and
$\lambda_{k+1}
<\lambda_k$; or
\\
(2) $(x_k,y_k)=(x_{k-1},y_{k-1})$ and
 the point $(\tilde x_{k+1},\tilde y_{k+1})$
is the outcome of one iteration (at $(x_k,y_k)$) of
Rockafellar's PMM for problem~\eqref{eq:np2.xi} with
$\tilde x=\tilde x_{k}$ and $\lambda=\lambda_{k+1}$.

Next we present our algorithm,
where $\mathcal{N}_\theta$, 
$f_{[\tilde x]},g_{[\tilde x]}$ and $\rho(y,\alpha)$ are 
as in \eqref{eq:viz}, 
\eqref{eq:qpfg},
and Definition~\ref{df:rho}, respectively.

\bigskip

\SetKwInput{KwInit}{initialization}

\LinesNumbered

\SetAlgoNlRelativeSize{0}
\begin{algorithm}[H]
\SetAlgoLined
\KwInit{choose $(x_0,y_0)=(\tilde x_1,\tilde y_1)\in \R^n\times\R^m_+$,
  $0<\sigma<1$,
  $0<\theta\leq 1/4$\;
 \quad define $h:=$  positive root of 
   $\theta(1+h')\left(1+h'\left(
    1+1/\sigma
    \right)\right)^2=1$,\,\,\,  $\tau=h/(1+h)$\;
\quad   choose $\lambda_1>0$ such that
  $(\tilde x_1,\tilde y_1)\in\mathcal{N}_{\theta^2}((x_0,y_0),\lambda_1)$
  and set $k\leftarrow 1$
\BlankLine
}
\eIf{$\rho(\tilde y_{k},\theta^2/\lambda_k)
     \leq
     \sigma\norm{(\tilde x_{k},\tilde y_{k})-(x_{k-1},y_{k-1})}$
}
{\BlankLine
  $\lambda_{k+1}:=(1-\tau)\lambda_k$\;
  $(\tilde x_{k+1},\tilde y_{k+1}):=(\tilde x_{k},\tilde y_{k})$\;
  $x_k:=x_{k-1}-\tau\lambda_k
        [\nabla f(\tilde x_k)+\nabla g(\tilde x_k)\tilde y_k]$,
  $y_k:=y_{k-1}+\tau[\lambda_kg(\tilde x_k)
        +(\lambda_kg(\tilde x_k)+y_{k-1})_-]$\;
\BlankLine
}
{\BlankLine
  $\lambda_{k+1}:=(1-\tau)^{-1}\lambda_k$\;
  $(x_k,y_k):=(x_{k-1},y_{k-1})$\;
  $(\tilde x_{k+1},\tilde y_{k+1}):=\arg\displaystyle\min_{x\in\R^n}\max_{y\in\R^m_+}
    \;\;f_{[\tilde x_{k}]}(x)+\Inner{y}{g_{[\tilde x_{k}]}(x)}
    +\dfrac{\norm{x-x_k}^2- \norm{y-y_k}^2
    }{2\lambda_{k+1}}$\;
}
set $k\leftarrow k+1$ and go to step 1\;
\caption{Relaxed hybrid proximal extragradient method 
of multipliers
based on 2nd ord.\ approx.\  (r-HPEMM-2o)}
\end{algorithm}

\bigskip

To simplify the presentation of Algorithm 1, 
we have omitted a stopping test.
First, we discuss its initialization.
In the absence of additional information on the dual variables
$y$, one shall consider 
the initialization
\begin{align}
  \label{eq:gi}
  (x_0,y_0)=(\tilde x_1,\tilde y_1)=(x,0),
\end{align}
where $x$ is ``close'' to the feasible set.
If $(x,y)\in\R^n\times\R^m_+$ an approximated 
solution of \eqref{eq:kkt.p} is available, 
one can do a ``warm start'' by
setting $(x_0,y_0)=(\tilde x_1,\tilde y_1)=(x,y)$.
Note that $h>0$ and $0<\tau<1$.
%
Existence of $\lambda_1>0$ as prescribed in 
this step follows from the
inclusion 
$(\tilde x_1,\tilde y_1)\in\R^n\times\R^m_+$ and from
Proposition~\ref{pr:boot}.
Moreover, if we compute $\lambda=\lambda_1>0$ 
satisfying
the inequality 
\[
 \left(\dfrac{2\norm{L_g}\norm{\So(x_0,y_0)}^2}{3}
\right)\lambda^3
+\left(\dfrac{L_0+\inner{L_g}{|y_0|}}{2}
\norm{\So(x_0,y_0)}\right)
\lambda^2-\theta^2\leq 0,
\]
where the operator $\So$ is defined 
in~\eqref{def:lag.e},
use Lemma~\ref{lm:ef} (c) and 
Definition~\ref{df:rho}
we find
$\sqrt{\Psi_{\So,(x_0,y_0),\lambda_1}(x_0,y_0)}\leq 
\lambda_1\norm{\So(x_0,y_0)}\leq \rho(y_0,
\theta^2/\lambda_1)$ which, in turn, combined with the fact that 
$(\tilde x_1,\tilde y_1)=(x_0,y_0)$, gives the inclusion 
in the initialization of Algorithm 1.

The computational cost of block of steps [\textbf{2},\textbf{3},\textbf{4}]
is negligible.
The initialization $\lambda_1>0$, together with the update of $\lambda_k$ by
step \textbf{2} or \textbf{6} guarantee that $\lambda_k>0$ for all $k$.
Therefore, the saddle-point problem to be solved in step \textbf{8} is
strongly convex-concave and hence has a unique solution.
The computational burden of the algorithm is in the computation of the
solution of this problem.

We will assume that $(\tilde x_1,\tilde y_1)$ {does not satisfy}
\eqref{eq:kkt.p}, i.e., the KKT conditions for \eqref{eq:np}, otherwise we would
already have a solution for the KKT
system and 
$\tilde x_1$ would be a
solution of \eqref{eq:np}.
For the sake of conciseness we introduce, for $k=1,\dots$, the notation
\begin{align}
  \label{eq:zks}
  z_{k-1}&=(x_{k-1},y_{k-1}), \qquad
  \tilde z_k= (\tilde x_k,\tilde y_k),\qquad
     \rho_k=\rho(\tilde y_k,\theta^2/\lambda_{k}).
\end{align}
Since there are two kinds of iterations in 
Algorithm 1, its is convenient to
have a notation for them. Define
\begin{align}
  \label{eq:AB}
  \begin{aligned}
    A&:=\{k\in\N\setminus\set{0}\;|\;
     \rho_{k} \leq 
     \sigma\norm{\tilde z_{k}-z_{k-1}}\},
     \;\;\;
    B:=\{k\in\N\setminus\set{0}\;|\;
    \rho_{k}>
    \sigma\norm{\tilde z_{k}-z_{k-1}}\}.
  \end{aligned}
\end{align}
Observe that in iteration $k$, either $k\in A$ and steps \textbf{2},
\textbf{3}, \textbf{4} are executed, or $k\in B$ and steps \textbf{6},
\textbf{7}, \textbf{8} are executed.

\begin{proposition}
  \label{pr:on.theta2}
  For $k=1,\dots$,
  \begin{enumerate}
	\item[\emph{(a)}] \label{item:a} $\tilde z_{k} \in
    \mathcal{N}_{\theta^2}(z_{k-1},\lambda_k)$;
  \item[\emph{(b)}] \label{item:b} $\sqrt{\Psi_{\So,z_{k-1},\lambda_k}(\tilde z_{k})}
    \leq \rho_{k}$.
  \item[\emph{(c)}]
    \label{item:new} $z_{k-1} \in\R^n\times\R^m_+$ .
   \end{enumerate}
\end{proposition}
\begin{proof}
  We will use induction on $k\geq 1$ 
 for proving (a).
  In view of the initialization of
  Algorithm 1, this inclusion holds
  trivially for $k=1$.
  Suppose that this inclusion holds for $k={k_0}$. We shall consider two
  possibilities.

  \medskip

  \noindent
  (i) ${k_0}\in A$:
  It follows from Proposition~\ref{pr:relax2} and the update rules in
  steps \textbf{2} and \textbf{4} that
  \begin{align*}
    \sqrt{\Psi_{\So,z_{k_0},\lambda_{{k_0}+1}}(\tilde z_{{k_0}})}\leq
    \sqrt{\Psi_{\So,z_{{k_0}-1},\lambda_{k_0}}(\tilde z_{{k_0}})}
    \leq \rho(\tilde y_{{k_0}},\theta^2/\lambda_{{k_0}})
    \leq \rho(\tilde y_{{k_0}},\theta^2/\lambda_{{k_0}+1})
   \end{align*}
   where the second inequality follows from the inclusion
   $\tilde z_{{k_0}}\in\mathcal{N}_{\theta^2}(z_{{k_0}-1},\lambda_{{k_0}})$
   and \eqref{eq:viz.2}; and the third inequality follows from step
   \textbf{2} and Definition~\ref{df:rho}.
   It follows from the above inequalities and 
   \eqref{eq:viz.2}
   that $\tilde z_{{k_0}}\in \mathcal{N}_{\theta^2}( z_{{k_0}},\lambda_{{k_0}+1})$.
 By step \textbf{3}, $\tilde z_{k_0+1}=\tilde z_{k_0}$. Therefore,
  the inclusion of Item (a) holds for $k={k_0}+1$ in case (i).

   \medskip

   \noindent
   (ii) ${k_0}\in B$: In this case, by step \textbf{7},
   $z_{{k_0}}=z_{{k_0}-1}$ and, using definition~\eqref{eq:AB}, the
   notation~\eqref{eq:zks}, and the assumption that the inclusion in
   Item (a) holds for $k={k_0}$ we conclude that
   \begin{align*}
     \norm{\tilde z_{{k_0}}-z_{{k_0}}}<\rho_{{k_0}}/\sigma,
     \quad
     \tilde z_{{k_0}}\in \mathcal{N}_{\theta^2}(z_{{k_0}},\lambda_{{k_0}}),
     \quad
     \sqrt{\Psi_{\So,z_{{k_0}},\lambda_{{k_0}}}(\tilde z_{{k_0}})} \leq
       \rho_{{k_0}}.
  \end{align*}
  Direct use of the definitions of $h$, $\tau$, and step \textbf{6}
  gives $\lambda_{{{k_0}}+1}=(1+h)\lambda_{{k_0}}$. Defining
  $\rho'=\sqrt{\Psi_{\So,z_{{k_0}},\lambda_{{k_0}+1}}
    (\tilde z_{{k_0}})}$,
  it follows from the above inequalities and from
  Proposition~\ref{pr:linc} that,
  \begin{align*}
  \rho'\leq
  (1+h)\sqrt{\Psi_{\So,z_{{k_0}},\lambda_{{k_0}}}
  (\tilde z_{{k_0}})}
  +h\norm{\tilde z_{{k_0}}-z_{{k_0}}}\leq
  (1+h(1+1/\sigma))\rho_{{k_0}}.
 \end{align*}
 Therefore,
 \begin{align*}
   \lambda_{{k_0}+1}\left(
   \dfrac{L_0+\inner{L_g}{|\tilde y_{{k_0}}|}}{2}+
   \dfrac{2\norm{L_g}}{3}\rho'\right)\rho'
   & \leq(1+h)\left(1+h\left(1+\dfrac{1}{\sigma}\right)
     \right)^2\\
   & \times
     \lambda_{{k_0}}\left(
     \dfrac{L_0+\inner{L_g}{|\tilde y_{{k_0}}|}}{2}+
     \dfrac{2\norm{L_g}}{3}
     \rho_{{k_0}}
     \right)\rho_{{k_0}}
   \\
   &
     =(1+h)\left(1+h\left(1+\dfrac{1}{\sigma}\right)\right)^2\theta^2=\theta,
 \end{align*}
   where
   we also have used Definition \ref{df:rho}
   and the definition of $h$ (in the initialization
   of Algorithm 1).
 %
  %
  It follows from the above inequality, 
  the definition of $\rho'$ and \eqref{eq:viz}  that
  \[
   \tilde z_{{k_0}}\in\mathcal{N}_\theta
   (z_{{k_0}},\lambda_{{k_0}+1}).
  \]
  Using this inclusion, step \textbf{8} and 
  Proposition~\ref{pr:rc}
  we conclude that the inclusion in Item (a) also holds
  for $k={k_0}+1$.

  Item (b) follows trivially from Item (a), 
  \eqref{eq:viz.2}  and \eqref{eq:zks}. 
  Item (c) follows from the fact that $y_0\geq 0$, 
  the definitions
  of steps \textbf{3}, \textbf{4}, and the last 
  part of Proposition~\ref{pr:relax2}.
\end{proof}

\section*{Algorithm 1 as a 
realization 
of the large-step r-HPE method}

In this subsection,
we will show that 
a subsequence generated by Algorithm 1
happens to be a sequence generated
by the large-step r-HPE method 
described in~\eqref{eq:rhpe}
for solving a monotone inclusion problem
associated with \eqref{eq:np}.
This result will be instrumental
for evaluating (in the next section)
the iteration-complexity of Algorithm 1.
In fact, we will prove that iterations 
with $k\in A$, where steps
\textbf{2}, \textbf{3}, \textbf{4} are 
executed, are large-step r-HPE iterations
for the monotone inclusion problem
\begin{align}
  \label{eq:mip2}
  0\in T(z):=\left(\So+\Nc_{\R^n\times\R^m_+}\right)(z),\qquad
  z=(x,y)\in\R^n\times\R^m,
\end{align}
where the operator $\So$ is defined in~\eqref{def:lag.e}.

Define, for $k=1,2,\dots$,
\begin{align}
\label{eq:s60}
w_k&=(g(\tilde x_k)+\lambda_k^{-1}y_{k-1})_-,\qquad v_k=\So(\tilde
z_k)-(0,w_k), \qquad 
\varepsilon_k=\inner{\tilde y_k}{w_k},
\end{align}
where $\tilde z_k$ is defined in \eqref{eq:zks}.
We will show that, whenever $k\in A$, 
the variables
$\tilde z_k$, $v_k$, and $\varepsilon_k$ 
provide an approximated
solution of the proximal inclusion-equation system
\begin{align*}
  v\in (\So+\Nc_{\R^n\times\R^m_+})(z),\qquad
  \lambda_k v+z-z_{k-1}=0,
\end{align*}
as required in the first line of~\eqref{eq:rhpe}.
We divided the proof of this fact in two parts, 
the next proposition 
and the subsequent lemma.
\begin{proposition}
\label{pr:ma}
  For $k=1,2,\ldots$,
  \begin{enumerate}
  \item[\emph{(a)}] $-(0,w_k)
   \in\partial_{\varepsilon_k}
        \delta_{\R^n\times\R^m_+}(\tilde z_k)=
  \Nc_{\R^n\times\R^m_+}^
  {[\varepsilon_k]}(\tilde z_k) $;
 \item[\emph{(b)}] $v_k\in(\So+\Nc_{\R^n\times\R^m_+}^{[\varepsilon_k]})
      (\tilde z_k)\subset
     (\So+\Nc_{\R^n\times\R^m_+})^{[\varepsilon_k]}(\tilde z_k)$;
 \item[\emph{(c)}] $\norm{ \lambda_kv_k+\tilde z_k-z_{k-1}}^2
    +2\lambda_k\varepsilon_k = 
    \Psi_{\So,z_{k-1},\lambda_k}(\tilde z_k) \leq \rho_k^2$.
  \end{enumerate}
\end{proposition}
\begin{proof}
 Items (a), (b) and the equality in Item (c) 
 follow from definitions \eqref{eq:zks}, 
\eqref{eq:s60} and Items (a) and (b) of 
Lemma~\ref{lm:loc3}.
The inequality in Item (c) follows from 
Proposition~\ref{pr:on.theta2}(b).
\end{proof}
%
Define
\begin{align}
  \begin{aligned}
    A_k&:=\{j\in \N\;|\;j\leq k,\; 
   \text{steps {\bf 2, 3, 4} are executed at 
  iteration $j$}\},
    \\
    B_k&:=\{j\in\N\;|\;j\leq k,\; 
   \text{steps {\bf 6, 7, 8} are executed at 
   iteration $j$} \}
  \end{aligned}
\end{align}
and observe that 
\[
 A=\bigcup_{k\in \N}A_k,\qquad B=\bigcup_{k\in \N}B_k. 
\] 

From now on,  $\# C$ stands for the number of elements of
a set $C$. 
To further simplify the converge analysis, define
\begin{align}
\label{eq:s70}
  I=\set{i\in\N\;|\; 1 \leq i \leq \# A},\quad
  k_0=0,\quad k_i=\text{$i$-th element of }A.
\end{align}
Note that $  k_0<k_1<k_2\cdots$, $A=\set{k_i\;|\; i\in I}$ and,
in view of \eqref{eq:AB} and step \textbf{7}
of Algorithm 1,
 \begin{align}\label{eq:a.is.hpe.fp}
        z_k=z_{k_{i-1}},\qquad
        \text{ for }\;\;k_{i-1}\leq k <k_i,\;\;     \forall i\in I.
  \end{align}
In particular, we have
\begin{align}
  \label{eq:kminuone}
  z_{k_i-1}=z_{k_{i-1}}\qquad \forall i \in I.
\end{align}

In the next lemma we show that
for indexes in the set $A$, Algorithm 1
generates a subsequence which can be regarded
as a realization of the large-step r-HPE method
described in \eqref{eq:rhpe}, for solving the problem
\eqref{eq:mip2}.

 \begin{lemma}
    \label{lm:hpe}
  The sequences $(z_{k_i})_{i\in I}$, 
 $(\tilde z_{k_i})_{i\in I}$, $(v_{k_i})_{i\in I}$,
  $(\varepsilon_{k_i})_{i\in I}$, 
  $(\lambda_{k_i})_{i\in I}$ are generated by
  a realization of the \emph{r-HPE method} described in \eqref{eq:rhpe}
  for solving \eqref{eq:mip2}, that is, $0 \in (\So+\Nc_{\R^n\times\R^m_+})(z)$, 
  in the following sense:
  for all $i\in I$, 
    \begin{align}\label{eq:a.is.hpe.lp}
      \begin{aligned}
        &v_{k_i}\in(\So+\Nc_{\R^n\times\R^m_+}^
        {[\varepsilon_{k_i}]})(\tilde z_{k_i})\subset
        (\So+\Nc_{\R^n\times\R^m_+})^{[\varepsilon_{k_i}]}
        (\tilde z_{k_i}),  \\
        &\norm{ \lambda_{k_i}v_{k_i} +\tilde
          z_{k_i}-z_{k_{i-1}}}^2+2\lambda_{k_i}
        \varepsilon_{k_i}\leq\rho_{k_i}^2 \leq \sigma^2
        \norm{\tilde z_{k_i}-z_{k_{i-1}}}^2,\\
        &z_{k_i}=z_{k_{i-1}}-\tau\lambda_{k_i}v_{k_i}.
      \end{aligned}
    \end{align}
  Moreover, if $I$ is finite and $i_{\mathrm{M}}:=\max I$ then
  $z_k=z_{k_{i_{\mathrm{M}}}}$ for $k\geq k_{i_{\mathrm{M}}}$.
\end{lemma}

\begin{proof}
The two inclusions in the first line 
of \eqref{eq:a.is.hpe.lp} 
follow trivially from Proposition~\ref{pr:ma}(b).
The first inequality in the second line of 
\eqref{eq:a.is.hpe.lp}
follows from \eqref{eq:kminuone} 
and 
Proposition~\ref{pr:ma}(c); the
second inequality follows from the 
inclusion $k_i\in A$, \eqref{eq:AB},
step \textbf{1} of Algorithm 1 and
\eqref{eq:kminuone}.
The equality in the last line of 
\eqref{eq:a.is.hpe.lp} follows
from the inclusion $k_i\in A$, 
\eqref{eq:AB} step \textbf{4} of Algorithm 1,
\eqref{eq:kminuone} and \eqref{eq:s60}.
Finally, the last statement of the 
lemma is a direct consequence of \eqref{eq:zks},
\eqref{eq:AB} and step \textbf{7} of Algorithm 1.
\end{proof}

As we already observed in Proposition~\ref{pr:rf.sdp}, 
\eqref{eq:mip2} and
\eqref{eq:kkt.p} are equivalent, in the sense that both problems have
the same solution set.
From now on we will use the notation $\mathcal{K}$ for this solution set,
that is,
\begin{align}
  \label{eq:def.k}
  \begin{aligned}
    \mathcal{K}
    &= \left(\So+\Nc_{\R^n\times\R^m_+}\right)^{-1}(0)\\
    &=\Set{(x,y)\in\R^n\times\R^m\;\left|\;\nabla f(x)+\nabla g(x)y=0,\;
    g(x) \leq 0,\; y \geq 0,\; \inner{y}{g(x)}=0
   }\right.\;.
  \end{aligned}
\end{align}
We assumed in (O.3) that this set is nonempty.
Let $z^*=(x^*,y^*)$ be the projection of $z_0=(x_0,y_0)$ onto $\sol$ and
$d_0$ the distance from $z_0$ to $\sol$,
\begin{align}
\label{eq:def.dz}
  z^*\in\sol,\quad d_0=\norm{z^*-z_0}=\min_{z\in\sol}\norm{z-z_0}.
\end{align}

To complement Lemma~\ref{lm:hpe}, we will prove that the
\emph{large-step condition} for the
large-step r-HPE method (stated in
Theorem~\ref{lm:rhpe2}) is satisfied 
for the realization of the method
presented in Lemma \ref{lm:hpe}.
Define
\begin{align}
\label{def:c.eta2}
  c&:=\dfrac{L_0+\inner{L_g}{|y_0|}}{2}+
     \left[\dfrac12+\dfrac{1/2+2\sigma/3}
     {\sqrt{1-\sigma^2}}
     \right]d_0\norm{L_g},\qquad \eta:=\dfrac{\theta^2}{\sigma c}.
\end{align}

\begin{proposition}
\label{pr:large}
Let $z^*\in \mathcal{K}$
and $d_0$, and $\eta$ as in \eqref{eq:def.dz},
and \eqref{def:c.eta2}, respectively.  
For all $i\in I$,
    \begin{align}
\label{eq:ineq.ls02}
 \norm{z^*-z_{k_i}}\leq d_0,\quad 
  \norm{z^*-\tilde z_{k_i}} \leq \dfrac{d_0}{\sqrt{1-\sigma^2}},\quad
      \norm{\tilde z_{k_i}-z_{k_{i-1}}} \leq 
   \dfrac{d_0}{\sqrt{1-\sigma^2}}.
    \end{align} 
   As a consequence, 
\begin{align}
\label{eq:lsc.2}
    \lambda_{k_i}
    \norm{\tilde z_{k_i}-z_{k_{i-1}}}\geq \eta.
    \end{align} 
\end{proposition}
\begin{proof}
Note first that \eqref{eq:ineq.ls02}
follows from Lemma \ref{lm:hpe},
items (c) and (d) of Proposition \ref{pr:a1},
\eqref{eq:kminuone} and
\eqref{eq:def.dz}. 
Using \eqref{eq:zks}, \eqref{eq:AB},
\eqref{eq:s70} and step \textbf{1} of
Algorithm 1 
we obtain
\begin{align*}
 \sigma\|\tilde z_{k_i}-z_{k_{i}-1}\|\geq 
\rho(\tilde y_{k_i},\theta^2/\lambda_{k_i})
\qquad \forall i\in I,
\end{align*}
which, in turn, combined with
the definition of $\rho(\cdot,\cdot)$
(see Definition~\ref{df:rho}) yields
\begin{align}
\label{eq:srh2}
 \left(
 \dfrac{L_0+\inner{L_g}{|\tilde y_{k_i}|}}{2}
 +\dfrac{2\norm{L_g}}{3}\sigma\norm{\tilde z_{k_i}-
 z_{k_{i}-1}}
 \right)\sigma\norm{\tilde z_{k_i}-
 z_{k_{i}-1}}\geq \dfrac{\theta^2}{\lambda_{k_i}}
\quad \forall i\in I.
\end{align}
Using the triangle inequality, 
\eqref{eq:def.dz} and the second inequality
in \eqref{eq:ineq.ls02} 
we obtain
\begin{align*}
\norm{z_0-\tilde z_{k_i}}\leq 
 d_0+\norm{z^*-\tilde z_{k_i}}\leq 
 d_0\left(1+\dfrac{1}{\sqrt{1-\sigma^2}}\right).
\end{align*}
Now, using the latter inequality, the fact that 
$\norm{z_0-\tilde z_{k_i}}\geq 
\norm{y_0-\tilde y_{k_i}}$ ($\forall i\in I$)
and the triangle inequality 
we find
\begin{align}
\nonumber 
\inner{L_g}{|\tilde y_{k_i}|}&\leq 
\norm{L_g}\norm{z_0-\tilde z_{k_i}}+
\inner{L_g}{|y_0|}\\
\label{eq:srh5}
&\leq d_0\norm{L_g}
\left(1+\dfrac{1}{\sqrt{1-\sigma^2}}\right)
+\inner{L_g}{|y_0|}\qquad \forall i\in I.
\end{align}
To finish the proof of \eqref{eq:lsc.2}, use 
\eqref{eq:kminuone}, substitute the terms in 
the right hand side of
the last inequalities in \eqref{eq:ineq.ls02}
and \eqref{eq:srh5} in the term 
inside the parentheses in \eqref{eq:srh2}
and use \eqref{def:c.eta2}.
\end{proof}

\section{Complexity  analysis}
\label{sec:rhpe}

In this section we study the pointwise
and ergodic iteration-complexity of Algorithm 1.
The main results are (essentially) a consequence of Lemma~\ref{lm:hpe} and
Proposition~\ref{pr:large} which guarantee that the (sub)sequences
$(z_{k_i})_{i\in I}$, $(\tilde z_{k_i})_{i\in I}$, $\dots$ can be regarded
as realizations of the large-step r-HPE method of Section~\ref{sec:sqp}, for
which pointwise and ergodic iteration-complexity results are known.

To study the ergodic iteration-complexity of Algorithm 1 we need to define
the ergodic sequences associated to $(\lambda_{k_i})_{i\in I}$,
$(\tilde z_{k_i})_{i\in I}$, $(v_{k_i})_{i\in I}$ and
$(\varepsilon_{k_i})_{i\in I}$, respectively (see \eqref{eq:d.eg}), namely
\begin{align}
\label{eq:ergodic.i}
  \begin{aligned}
    &\Lambda_i:=\tau\sum_{j=1} ^i\lambda_{k_j},\\
    &\tilde z_i^{\,a}=(\tilde x_i^a,
\tilde y_i^a):=
    \frac{1}{\;\Lambda_i}\;\tau\sum_{j=1}^i\lambda_{k_j} 
\tilde z_{k_j},
    \qquad\quad
    v_i^{\,a}:= \frac{1}{\;\Lambda_i}\;\tau\sum_{j=1}^i
    \lambda_{k_j} v_{k_j} ,
    \\
    &\varepsilon_i^{\,a}:=
    \frac{1}{\;\Lambda_i}\;\tau\sum_{j=1}^i\lambda_{k_j} (\varepsilon_{k_j}
    +\inner{\tilde z_{k_j}-\tilde z_i^{\,a}}{v_{k_j}-v_i^{\,a}}).
  \end{aligned}
\end{align}
Define also
\begin{align}
  \label{eq:lg2}
   \overline{\mathscr{L}}(x,y):=
  \begin{cases}
    f(x)+\inner{y}{g(x)},& y\geq 0\\
    -\infty,&\text{otherwise.}
  \end{cases}
\end{align}

Observe that
that 
a pair $(x,y)\in \mathcal{K}$, i.e., it is a solution 
of the KKT system  
\eqref{eq:kkt.p}
if and only if 
 $(0,0)\in \partial(\overline{\mathscr{L}}(\cdot,y)
    -\overline{\mathscr{L}}(x,\cdot))(x,y)$.
Since \eqref{eq:mip2} and \eqref{eq:kkt.p}
are equivalent, the latter observation
leads us to consider in this section
the notion of approximate
solution for \eqref{eq:mip2} which consists
in: for given tolerances $\overline \delta>0$
and $\overline\varepsilon>0$ find 
$((x,y),v,\varepsilon)$ such that
\begin{align}
\label{eq:def.aps}
 v\in \partial_\varepsilon(\overline{\mathscr{L}}(\cdot,y)
    -\overline{\mathscr{L}}(x,\cdot))(x,y),
\quad \norm{v}\leq \overline\delta,\quad 
\varepsilon\leq \overline\varepsilon.
\end{align} 

We will also consider as approximate solution
of \eqref{eq:mip2} any triple 
$((x,y),(p,q),\varepsilon)$ 
such that $\norm{(p,q)}\leq \overline\delta$,
$\varepsilon\leq \overline\varepsilon$ and 
\begin{align}
\label{eq:def.aps2}
 p=\nabla f(x)+\nabla g(x)y,\quad  g(x)+q\leq 0,
 \quad y\geq 0,\quad \inner{y}{g(x)+q}=-\varepsilon
\end{align} 
or
\begin{align}
\label{eq:def.aps3}
 p\in \partial_{x,\varepsilon'}
 \overline{\mathscr{L}}(x,y),\quad  g(x)+q\leq 0,
 \quad y\geq 0,\quad \inner{y}{g(x)+q}
 \geq -\varepsilon,
\end{align} 
where $\varepsilon':=\varepsilon+\inner{y}{g(x)+q}$.

It is worthing to compare the latter two
conditions with \eqref{eq:kkt.p} and also 
note that whenever
$\varepsilon'=0$ then \eqref{eq:def.aps3} reduces
to \eqref{eq:def.aps2}, that is, the latter condition
is a special case of \eqref{eq:def.aps3}.
Moreover, as Theorems \ref{th:pt.com}
and \ref{th:erg.com} will show, \eqref{eq:def.aps2}
and \eqref{eq:def.aps3} are related to the pointwise
and ergodic iteration-complexity of Algorithm 1, respectively. 

We start by studying 
rates of convergence of
Algorithm 1.

\begin{theorem}
\label{pr:conv.rate}
Let $(\tilde z_{k_i})_{i\in I}=
((\tilde x_{k_i},\tilde y_{k_i}))_{i\in I}$, 
$(v_{k_i})_{i\in I}$ 
and $(\varepsilon_{k_i})_{i\in I}$ be
(sub)sequences generated by \emph{Algorithm 1}
where the the set of indexes $I$ is defined
in \eqref{eq:s70}.  
Let also $(\tilde z_i^a)_{i\in I}=
((\tilde x_i^a,\tilde y_i^a))_{i\in I}$,  
$(v_i^a)_{i\in I}$
and 
$(\varepsilon_i^a)_{i\in I}$ be
as in~\eqref{eq:ergodic.i}.  
Then, for any $i\in I$,
 \begin{enumerate}
  \item[\emph{(a)}]\emph{{\bf [pointwise]}}
 \label{it:rhpe2-10}
    there exists $j\in \{1,\dots, i\}$
    such that 
\begin{align}
\label{eq:kkt-p}
 v_{k_j}\in \partial_{\varepsilon_{k_j}}\,
\left(\;\overline{\mathscr{L}}(\cdot, \tilde y_{k_j})-
\overline{\mathscr{L}}(\tilde x_{k_j},\cdot)\right)
(\tilde x_{k_j},\tilde y_{k_j})
\end{align}
and
\begin{align}
\label{eq:kkt-p-r} 
      \Norm{v_{k_j}}
      \leq\dfrac{d_0^2}{i\tau(1-\sigma)\eta},
      \qquad
      \varepsilon_{k_j}\leq
      \dfrac{\sigma^2d_0^3}
     {(i\tau)^{3/2}(1-\sigma^2)^{3/2}2\eta};
    \end{align}
\item[\emph{(b)}]\emph{{\bf [pointwise]}}
 there exists $j\in \{1,\dots, i\}$
 and $(p_j,q_j)\in \R^n\times \R^m$    
such that
 \begin{align}
 \label{eq:kkt-pp}
  \begin{aligned}
  &p_j=\nabla f(\tilde x_{k_j})+
  \nabla g(\tilde x_{k_j})\tilde y_{k_j}\,,\\
  &g(\tilde x_{k_j})+q_j\leq 0,
 \quad \tilde y_{k_j}\geq 0,
 \quad  \inner{\tilde y_{k_j}}{g(\tilde x_{k_j})+q_j}=-\varepsilon_{k_j}
  \end{aligned}
 \end{align}   
and
\begin{align}
\label{eq:kkt-p-rr} 
      \Norm{
       (p_j,q_j) 
       }
      \leq\dfrac{d_0^2}{i\tau(1-\sigma)\eta},
      \qquad
      \varepsilon_{k_j}\leq
      \dfrac{\sigma^2d_0^3}
     {(i\tau)^{3/2}(1-\sigma^2)^{3/2}2\eta};
    \end{align}

\item [\emph{(c)}]\emph{{\bf [ergodic]}}
\label{it:rhpe2-20}
   we have
   \begin{align}
    \label{eq:kkt-e}
    v_i^a\in \partial_{\varepsilon_{i}^a}\,
 \left(\;\overline{\mathscr{L}}(\cdot, \tilde y_{i}^a)- 
\overline{\mathscr{L}}(\tilde x_{i}^a,\cdot)\right)
(\tilde x_{i}^a,\tilde y_{i}^a) 
  \end{align}     
 and
    \begin{align}
     \label{eq:kkt-e-r}
      \norm{v_i^a}\leq 
      \dfrac{2d_0^2}{(i\tau)^{3/2}
       (\sqrt{1-\sigma^2})\eta},
      \qquad
\varepsilon_i^a\leq
 \dfrac{2d_0^3}{(i\tau)^{3/2}(1-\sigma^2)\eta};
    \end{align}
\item[\emph{(d)}]\emph{{\bf[ergodic]}}
  there exists 
    $(p_i^a,q_i^a)\in \R^n\times \R^m$ 
    such that
   \begin{align}
   \label{eq:spq}
	   \begin{aligned}
    &p_i^a\in \partial_{x,\varepsilon'_i}\,
   \overline{\mathscr{L}}
     (\tilde x_i^a,\tilde y_i^a),\\ 
     &g(\tilde x_i^a)+q_i^a\leq 0,\quad \tilde y_i^a\geq 0,\quad 
    \inner{\tilde y_i^a}{g(\tilde x_i^a)+q_i^a}\geq 
    -\varepsilon_i^a  
    \end{aligned}
	 \end{align}
and
    \begin{align}
     \label{eq:kkt-e-rr}
      \norm{(p_i^a,q_i^a)}\leq 
    \dfrac{2d_0^2}{(i\tau)^{3/2}
       (\sqrt{1-\sigma^2})\eta},
      \qquad
\varepsilon_i^a\leq
 \dfrac{2d_0^3}{(i\tau)^{3/2}(1-\sigma^2)\eta},
    \end{align}
 where $\varepsilon'_i:=\varepsilon_i^a+
 \inner{\tilde y_i^a}{g(\tilde x_i^a)+q_i^a}$. 

\end{enumerate}
\end{theorem}
\begin{proof}
  We first prove Items (a) and (c). Using Lemma~\ref{lm:hpe}, the last
  statement in Proposition~\ref{pr:large} and \eqref{eq:mip2} we have that
  Items (a) and (b) of Theorem~\ref{lm:rhpe2} hold for the sequences
  $(\tilde z_{k_i})_{i\in I}$, $(v_{k_i})_{i\in I}$ and
  $(\varepsilon_{k_i})_{i\in I}$. As a consequence, to finish the proof of
  Items (a) and (c) of the theorem, it remains to prove the inclusions
  \eqref{eq:kkt-p} and \eqref{eq:kkt-e}.
To this end, note first that 
from the equivalence between Items (a) and (c) of 
Proposition \ref{pr:a7} 
(with $\varepsilon'=0$)
we have the following equivalence for all $i\in I$
\begin{align*}
 v_{k_i} \in(\So+\Nc_{\R^n\times\R^m_+}^
        {[\varepsilon_{k_i}]})
 (\tilde x_{k_i},\tilde y_{k_i}) \iff
 v_{k_i}\in \partial_{\varepsilon_{k_i}}\left(
 \;\overline{\mathscr{L}}(\cdot,\tilde y_{k_i})-
 \overline{\mathscr{L}}(\tilde x_{k_i},\cdot)\right)
 (\tilde x_{k_i},\tilde y_{k_i}).
\end{align*} 
Hence, using the latter
equivalence, the first inclusion in (the first line)
of \eqref{eq:a.is.hpe.lp}, the inclusion in 
Theorem \ref{lm:rhpe2}(a), the remark after the latter
theorem, and \eqref{eq:mip2} we 
obtain \eqref{eq:kkt-p}.
%
Likewise, using an analogous reasoning 
and Proposition \ref{pr:sp0} we
also obtain \eqref{eq:kkt-e},
which finishes the proof of Items (a) and (c).

We claim that Item (b) follows from Item (a). 
Indeed, 
letting
$(p_i,q_i):=v_{k_i}$ (for all $i\in I$), 
using the definition of $v_{k_j}$ and 
$\varepsilon_{k_j}$ 
in \eqref{eq:s60},
the definition of $\So$ in \eqref{def:lag.e}
and the equivalence between Items (a) and (b)
of Proposition \ref{pr:a7} 
(with $\varepsilon'=0$) 
we obtain 
that $(p_j,q_j):=v_{k_j}$
satisfies \eqref{eq:kkt-pp} and \eqref{eq:kkt-p-rr}.
Using an analogous reasoning
we obtain that Item (d) follows from Item (c). 
\end{proof}

Next we analyze the sequence generated by 
Algorithm 1 for the set of indexes
$k\in B$.
Direct use of Algorithm 1's definition shows that
\begin{align}
\label{eq:form.lambda}
   \lambda_{k+1}=
  \left(\dfrac{1}{1-\tau}\right)^{\#B_k-\#A_k}\lambda_1
  \qquad \forall k\geq 1.
\end{align}
Define
\begin{align}
\label{def:rhobar}
  \overline{\rho}
   =\dfrac{2\theta^2}
  {\lambda_1\left(\dfrac{L_0}{2}+
  \sqrt{\left(\dfrac{L_0}{2}\right)^{2}
     +\dfrac{8\norm{L_g}\theta^2}{3\lambda_1}}\right)}.
\end{align}

In the next proposition we obtain a 
rate of convergence result for the sequence
generated by Algorithm 1 with $k\in B$.

\begin{proposition}
\label{pr:geo.rate}
Let $\rho_k$ for all $k\geq 1$ be as in
\eqref{eq:zks} and let also $\bar\rho>0$ be as in 
\eqref{def:rhobar}.
Then, for all $k\in B$,
\begin{align*}
 v_k\in \partial_{\varepsilon_k}
   \left(\;\overline{\mathscr{L}}(\cdot,\tilde y_k)
   -\overline{\mathscr{L}}(\tilde x_k,\cdot)\right)
    (\tilde x_k,\tilde y_k)
\end{align*}
and
\begin{align*}
\norm{v_k}\leq \dfrac{(1+1/\sigma)\rho_k}
     {\lambda_k},\qquad 
      \varepsilon_k\leq\dfrac{\rho_k^2}{2\lambda_k}.
\end{align*}
Moreover, if $\lambda_k\geq \lambda_1$ then 
$\rho_k\leq \overline{\rho}$.
\end{proposition}
\begin{proof}
First note that the desired inclusion follows from
Proposition \ref{pr:ma}(b) and the equivalence
between items (a) and (c) of Proposition \ref{pr:a7}.  
Moreover, by Proposition~\ref{pr:ma} (c) we have
\begin{align*}
 \norm{\lambda_kv_k+\tilde z_k-z_{k-1}}^2+
 2\lambda_k\varepsilon_k\leq \rho_k^2
 \qquad \forall k\geq 1.
\end{align*}
Note that the desired bound on $\varepsilon_k$ is a direct
consequence of the latter inequality. 
Moreover, this inequality
combined with the definition of $B$
(see \eqref{eq:AB}) gives
$\norm{\lambda_k v_k}\leq 
\norm{\lambda_k v_k+\tilde z_k-z_{k-1}}+
\norm{\tilde z_k-z_{k-1}}\leq (1+1/\sigma)\rho_k$ for
all $k\in B$,
which proves the desired bound on $\norm{v_k}$. 

Assume now that $\lambda_k\geq \lambda_1$.
Using Definition \ref{df:rho} and 
\eqref{eq:zks} we obtain for all $k\geq 1$
\[
 \rho_k=\rho(\tilde y_k,\theta^2/\lambda_k)=
\dfrac{2\theta^2}
{
\lambda_k\left(\dfrac{L_0+\inner{L_g}
{|\tilde y_k|}}{2}+
\sqrt{\left(\dfrac{L_0+\inner{L_g}
{|\tilde y_k|}}{2}\right)^2+\dfrac{8
\norm{L_g}\theta^2}{3\lambda_k}}\right)
},
\]
which, in turn, combined with \eqref{def:rhobar},
the assumption that $\lambda_k\geq \lambda_1$ and
the fact that $\inner{L_g}{|\tilde y_k|}\geq 0$
gives $\rho_k\leq \bar\rho$.
\end{proof}

Next we present the two main results of this paper, namely,
the pointwise and ergodic iteration-complexities
of Algorithm 1.

\begin{theorem}[pointwise iteration-complexity]
\label{th:pt.com}
For given tolerances 
$\overline\delta>0$ and 
$\overline\varepsilon>0$, after at most
\begin{align}
\nonumber
M&:=2\left\lceil\max\left\{
\dfrac{d_0^2}{\overline\delta \tau(1-\sigma)\eta},\;
 \dfrac{\sigma^{4/3}d_0^2}
{\overline\varepsilon^{2/3}\tau(1-\sigma^2)
(2\eta)^{2/3}}
    \right\}\right\rceil\\
\label{eq:pt.comp}
&\qquad    +
  \left\lceil
  \dfrac{
  \max\left\{
    \log^+ 
 \left((1+1/\sigma)\overline{\rho}/
  (\overline\delta \lambda_1)\right),
  \log^+ \left(\overline{\rho}^2/
  (2\overline\varepsilon\lambda_1)\right)
    \right\}}{\log (1/(1-\tau))}\right\rceil
  \end{align}
iterations, \emph{Algorithm 1}
finds $((x,y),v,\varepsilon)$
satisfying \eqref{eq:def.aps}
with the property that $((x,y),(p,q),\varepsilon)$
where $(p,q):=v$ also satisfies
%
\begin{align}
\label{eq:ap.kkt6}
\begin{aligned}
&p=\nabla f(x)+\nabla g(x)y,\quad
g(x)+q\leq 0,\quad
y\geq 0,\quad
\inner{y}{g(x)+q}=-\varepsilon,\\
& \norm{(p,q)}\leq \overline\delta,\;
\varepsilon\leq \overline\varepsilon.
\end{aligned}
\end{align}
\end{theorem}
\begin{proof}
First define
\begin{align}
\label{eq:def.mm}
  M_1:=\left\lceil\max\left\{
\dfrac{d_0^2}{\overline\delta \tau(1-\sigma)\eta},\;
 \dfrac{\sigma^{4/3}d_0^2}
{\overline\varepsilon^{2/3}\tau(1-\sigma^2)
(2\eta)^{2/3}}
    \right\}\right\rceil\quad \mbox{and}\qquad
  M_2:=M-2M_1.
 \end{align}
The proof is divided in two cases:
(i) $\#A\geq M_1$ and (ii) $\#A<M_1$ . 
In the first case,
the existence of $((x,y),v,\varepsilon)$ 
(resp. $((x,y),(p,q),\varepsilon)$)
satisfying \eqref{eq:def.aps}
(resp. \eqref{eq:ap.kkt6})
in at most $M_1$ iterations
follows from Theorem \ref{pr:conv.rate}(a)
(resp. Theorem \ref{pr:conv.rate}(b)).
Since $M=2M_1+M_2\geq M_1$, it follows that, in this case,
the number of iterations is not bigger than $M$.

Consider now the case (ii)
and let $k^*\geq 1$ be such that 
$\#A=\#A_{k^*}=\#A_k$ for all $k\geq k^*$. 
As a consequence of the latter property
and the fact that $\#A<M_1$ we
conclude that if $\#B_k\geq M_1+M_2$, 
for some $k\geq k^*$, 
then 
\begin{align}
 \label{eq:def.mz}
 \beta_k:=\#B_k-\#A_k\geq \#B_k-M_1\geq M_2. 
\end{align}
Using the latter inequality, \eqref{eq:pt.comp}
and \eqref{eq:def.mm} we find 
\[
 \beta_k \geq M_2\geq 
\dfrac{
  \max\left\{
    \log^+ 
 \left((1+1/\sigma)\overline{\rho}/
  (\overline\delta \lambda_1)\right),
  \log^+ \left(\overline{\rho}^2/
  (2\overline\varepsilon\lambda_1)\right)
    \right\}}{\log (1/(1-\tau))},
\]
which is clearly equivalent to
\begin{align}
\label{eq:log2}
 &\log\left(
  \left(\dfrac{1}{1-\tau}\right)^{\beta_k}\lambda_1\right)
 +\log\left(\dfrac{1}{(1+1/\sigma)\bar\rho}\right)
\geq \log\left(\dfrac{1}{\bar\delta}\right),\\
\label{eq:log3}
& \log\left(
  \left(\dfrac{1}{1-\tau}\right)^{\beta_k}\lambda_1\right)
 +\log\left(\dfrac{2}{\bar\rho^2}\right)
\geq \log\left(\dfrac{1}{\bar\varepsilon}\right).
\end{align}
Now using the definition in \eqref{eq:def.mz}, 
\eqref{eq:log2} (resp. \eqref{eq:log3}) and \eqref{eq:form.lambda}
we obtain  
$\log(\lambda_k/[(1+1/\sigma)\bar\rho])
\geq \log(1/\bar\delta)$
(resp. $\log(2\lambda_k/\bar\rho^2)\geq 
\log(1/\bar \varepsilon)$) which yields 
\[
 \dfrac{(1+1/\sigma)\bar\rho}{\lambda_k}\leq 
 \bar\delta
 \qquad \left(\mbox{resp.}\;\; 
  \dfrac{\bar\rho^2}{2\lambda_k}\leq \bar\varepsilon
 \right).
\]
It follows from the latter inequality and
Proposition \ref{pr:geo.rate}
that $((x,y),v,\varepsilon):=
((\tilde x_k,\tilde y_k),v_k,\varepsilon_k)$
satisfies \eqref{eq:def.aps} and, due to 
Proposition \ref{pr:a7}, that
$((x,y),(p,q),\varepsilon):=
((\tilde x_k,\tilde y_k),v_k,\varepsilon_k)$
satisfies \eqref{eq:ap.kkt6}. 
Since the index $k$ has been chosen
to satisfy $\#A_k<M_1$ and $\#B_k\geq M_1+M_2$ we conclude that
the total number of iterations
is at most $M_1+(M_1+M_2)=M$. 
\end{proof}


\begin{theorem}[ergodic iteration-complexity]
\label{th:erg.com}
For given tolerances 
$\overline\delta>0$ and 
$\overline\varepsilon>0$, after at most
\begin{align}
 \nonumber
\widetilde {M}&:=2\left\lceil\max\left\{
\dfrac{2^{2/3}d_0^{4/3}}
{\overline\delta^{2/3}\tau 
 \left(\eta\sqrt{1-\sigma}\right)^{2/3}},\;
 \dfrac{2^{2/3}d_0^2}
{\overline\varepsilon^{2/3}\tau
\left(\eta(1-\sigma^2)\right)^{2/3}}
    \right\}\right\rceil\\
\label{eq:erg.comp}
&\qquad    +
  \left\lceil
  \dfrac{
  \max\left\{
    \log^+ 
 \left((1+1/\sigma)\overline{\rho}/
  (\overline\delta \lambda_1)\right),
  \log^+ \left(\overline{\rho}^2/
  (2\overline\varepsilon\lambda_1)\right)
    \right\}}{\log (1/(1-\tau))}\right\rceil
  \end{align}
iterations, \emph{Algorithm 1}
finds $((x,y),v,\varepsilon)$
satisfying \eqref{eq:def.aps}
with the property that 
$((x,y),(p,q),\varepsilon)$
where $(p,q):=v$
also satisfies
\begin{align}
\label{eq:ap.kkt8}
\begin{aligned}
&p\in \partial_{x,\varepsilon'}
\overline{\mathscr{L}}(x,y),\
\quad
g(x)+q\leq 0,\quad
y\geq 0,\quad
\inner{y}{g(x)+q}\geq -\varepsilon,\\
& \norm{(p,q)}\leq \overline\delta,\;
\varepsilon\leq \overline\varepsilon,
\end{aligned}
\end{align}
where $\varepsilon':=\varepsilon+\inner{y}{g(x)+q}$.
\end{theorem}
\begin{proof}
 The proof follows the same outline
of Theorem \ref{th:pt.com}'s proof. 
\end{proof}

\appendix

\section{Appendix}

\begin{proposition}
\label{pr:a7}
 Let $(\tilde x,\tilde y)\in 
 \R^n\times \R^m_+$, 
 $v=(p,q)
  \in \R^n\times \R^m$
 and $\varepsilon\geq 0$ be given and define
 \begin{align} 
  \label{eq:vele}
w:=-(g(\tilde x)+q),\qquad  
\varepsilon':= \varepsilon-\inner{\tilde y}{w}.
 \end{align}
 The following conditions are equivalent:
 \begin{itemize}
 \item[\emph{(a)}]
   $
   v\in \partial_\varepsilon 
  \left(\overline{\mathscr{L}}(\cdot,\tilde y)
 -\overline{\mathscr{L}}(\tilde x,\cdot)
  \right)(\tilde x,\tilde y);
    $ 
 \item[\emph{(b)}]
 $
  w\geq 0,\quad  \inner{\tilde y}{w}\leq \varepsilon,
  \quad p\in \partial_{x,\varepsilon'}
 \overline{\mathscr{L}}(\tilde x,\tilde y);
 $ 
\item[\emph{(c)}] $0\leq \varepsilon'\leq\varepsilon$
 and
 $-w\in N_{\R^m_+}^{[\varepsilon-\varepsilon']}
 (\tilde y)$,\;\; 
 $p\in \partial_{x,\varepsilon'}\overline{\mathscr{L}}
  (\tilde x,\tilde y)$.
 \end{itemize}
\end{proposition}
\begin{proof}
$(a)\iff (b)$. Note that the inclusion in (a)
is equivalent to 
\begin{align}
\label{eq:744}
 \overline{\mathscr{L}}(x,\tilde y)-
 \overline{\mathscr{L}}(\tilde x,y)\geq
 \inner{p}{x-\tilde x}+\inner{q}{y-\tilde y}-\varepsilon
\qquad \forall (x,y)\in \R^n\times \R^m_+, 
\end{align}
which, in view of \eqref{eq:lg2}
and \eqref{eq:vele}, is equivalent to
\begin{align*}
 \overline{\mathscr{L}}(x,\tilde y)-
 \overline{\mathscr{L}}(\tilde x,\tilde y)
+\inf_{y\geq 0}\,\inner{w}{y}
\geq
 \inner{p}{x-\tilde x}-\varepsilon' \qquad \forall x\in \R^n.
\end{align*}
The latter inequality is clearly equivalent to (b).

$(b)\iff (c)$. Using \eqref{eq:npe}, the fact
that $\tilde y\geq 0$
and the definition of $\varepsilon'$ in \eqref{eq:vele}
we obtain that  
the first inequality in (b) is equivalent to
$\varepsilon'\leq \varepsilon$ and the second inequality
in (c). To finish the proof note that the second
inequality in (b) is equivalent to $\varepsilon'\geq 0$.
\end{proof}

\begin{proposition}
\label{pr:sp0}
Let $X\subset \R^n$ and $Y\subset \R^m$ be given convex sets and $\Gamma:X\times Y\to \R$ be a function such that, for each $(x,y)\in X\times Y$,
the function $\Gamma(\cdot,y)-\Gamma(x,\cdot):X\times Y\to \R$ is convex. Suppose that, for $j=1,\dots,i$,
$(\tilde{x}_j,\tilde{y}_j)\in X\times Y$ and 
$(p_{j},q_{j})\in \R^n\times \R^m$ satisfy
\begin{align*}
 (p_{j},q_{j})\in \partial_{\varepsilon_j}
\left(\Gamma(\cdot,\tilde{y}_j)-\Gamma(\tilde{x}_j,\cdot)\right)
(\tilde{x}_j,\tilde{y}_j)\,.
\end{align*}
Let $\alpha_1,\cdots,\alpha_i\geq 0$ be such that 
$\sum_{j=1}^i\,\alpha_j=1$, and define
\begin{align*}
 (\tilde{x}_i^a,\tilde{y}_i^a)=\sum_{j=1}^i\,\alpha_j (\tilde{x}_j,\tilde{y}_j),\quad (p_{i}^a,q_{i}^a)=\sum_{j=1}^i\,\alpha_j (p_{j},q_{j}),
\end{align*}
\begin{align*}
 \varepsilon_i^a=\sum_{j=1}^i\,\alpha_j
\left[\varepsilon_j+\inner{\tilde{x}_j-\tilde{x}_i^a}
{p_{j}}+\inner{\tilde{y}_j-\tilde{y}_i^a}{q_{j}}\right]\,.
\end{align*}
Then, $\varepsilon_i^a\geq 0$ and
\begin{align*}
 (p_{i}^a,q_{i}^a)\in \partial_{\varepsilon_i^a}
\left(\Gamma(\cdot,\tilde{y}_i^a)-\Gamma(\tilde{x}_i^a,\cdot)\right)
(\tilde{x}_i^a,\tilde{y}_i^a)\,.
\end{align*}
\end{proposition}
\begin{proof}
 See~\cite[Proposition 5.1]{MonSva10-1}\,.
\end{proof}


%
%

\end{document}